\documentclass[a4paper,11pt]{article}
\usepackage{inputenc}
\usepackage{authblk}
\usepackage[table]{xcolor}
\usepackage{color}
\usepackage{amsmath}
\usepackage{amssymb}
\usepackage{amsthm}
\usepackage{booktabs}
\usepackage{xcolor,graphicx,float}
\usepackage{hyperref}
\usepackage{soul}
\usepackage[numbers,sort&compress]{natbib}
\usepackage{indentfirst}
\hypersetup{colorlinks=true,
	linkcolor=blue,
	anchorcolor=blue,
	citecolor=blue}
\usepackage{enumitem}
\setlist[itemize]{itemsep=0pt, topsep=0pt}
\usepackage[
a4paper,
textwidth=16cm,
textheight=23cm,
centering
]{geometry}

\theoremstyle{plain}
\newtheorem{theorem}{\bf Theorem}[section]
\newtheorem{Lemma}[theorem]{\bf Lemma}
\newtheorem{proposition}[theorem]{\bf Proposition}

\newtheorem{conjecture}{\bf Conjecture}

\theoremstyle{remark}

\newtheorem{remark}{\bf Remark}

\numberwithin{equation}{section}

\title{\bf\Large On the Frankl--Tokushige conjecture and almost complete $r$-cross $t$-intersection theorems for vector spaces}
\author[1]{Yao Li\thanks{E-mail: \text{yaoli@mail.bnu.edu.cn}}}
\author[1]{Benjian Lv\thanks{Corresponding author. E-mail: \text{bjlv@bnu.edu.cn}}}
\author[1]{Jie Wen\thanks{ E-mail: \text{jwen@mail.bnu.edu.cn}}}
\affil[1]{\small Laboratory of Mathematics and Complex Systems (Ministry of Education), School of Mathematical Sciences, Beijing Normal University, Beijing 100875, China}
\date{}

\begin{document}
	\renewcommand{\baselinestretch}{1.2}
	\maketitle
	\begin{abstract}
Let $r\geq3$ and $k_1\geq k_2\geq\cdots\geq k_r\geq t$. Let $\mathcal{F}_1,\mathcal{F}_2,\ldots,\mathcal{F}_r$ be families of subspaces, of respective dimensions $k_1,k_2,\ldots,k_r$, in an $n$-dimensional vector space over the finite field $\mathbb{F}_q$. The $r$ families are called $r$-cross $t$-intersecting if $\dim \left(F_{1} \cap F_{2} \cap \cdots \cap F_{r}\right) \geq t$ for all $F_{i} \in \mathcal{F}_{i}, i = 1,2,\dots,r$. In 2016, Frankl and Tokushige conjectured that $\prod_{i=1}^{r}|\mathcal{F}_i|\leq\prod_{i=1}^{r}{n-1\brack k_i-1}$ for $t=1$ and $n\geq rk_1/(r-1)$. The appealing conjecture suggests establishing intersection theorems for $n\sim ck_1$ with $c=c(r)\in(1,2)$, a direction that has long been challenging.
		
In this paper, we overcome this barrier by proving that
$$\prod_{i=1}^{r}|\mathcal{F}_i|\leq\prod_{i=1}^{r}{n-t\brack k_i-t}\;\;\mbox{for all}\;\;t\geq1\;\mbox{and}\;n\geq rk_1/(r-1)+C(t,r),$$
where $C(t,r)=rt/(r-1)+1$. This proves the Frankl--Tokushige conjecture except for at most three values of $n$, and establishes an Erd\H{o}s--Ko--Rado type theorem for almost all values of parameters. Furthermore, we characterize all extremal configurations. Our proof is purely combinatorial and based on the $t$-cover method, with several essential refinements. We also obtain almost complete intersection theorems for $r$-wise $t$-intersecting families and non-trivial $r$-cross $t$-intersecting families.

		
		
	\end{abstract}
\section{Introduction}
The seminal Erd\H{o}s--Ko--Rado theorem \cite{EKR} initiated the study of intersection problems, which has since developed into a longstanding and active area of research in extremal set theory. The notion of intersection arises naturally for many combinatorial objects (see, e.g., \cite{Ellis-2011,Ellis-2012,Erdos-Szekely-2000,Frankl-Furedi,Frankl-Kupavskii-2018,Frankl-Wilson-1986,Ihringer,Ku-Renshaw-2008,Meagher-Sin,Moon,Scott}), and their study has led to a variety of works. For a systematic
introduction to intersection problems, we refer the reader to the surveys \cite{Frankl-2016,Ellis-book} and the monographs \cite{Godsil-Meagher,Frankl-2018}. 

In this paper, we study intersection problems for vector spaces. We first introduce some notation. Write $[r]:=\{1,2,\ldots,r\}$ the standard $r$-set. Let $V$ be an $n$-dimensional vector space over the finite field $\mathbb{F}_q$, where $q$ is a prime power. Let ${V\brack k}$ denote the set of all $k$-dimensional subspaces of $V$. We abbreviate `$k$-dimensional subspace' as `$k$-subspace'. The \emph{Gaussian binomial coefficient} is defined as
$$
{a\brack b} = \prod_{i=0}^{b-1}\frac{q^{a-i}-1}{q^{b-i}-1}.
$$
For convenience, we suppress the dependence on $q$ in the notation, and set ${a\brack 0}=1$ and ${a\brack c} =0$ if $c$ is negative. It is well known that ${n\brack k}$ counts the size of ${V\brack k}$. A family $\mathcal{F}\subseteq{V\brack k}$ is called \emph{$t$-intersecting} if $\dim(A\cap B)\geq t$ for all $A,B\in\mathcal{F}$. We say that $r$ families $\mathcal{F}_{1} \subseteq 
{ V \brack k_{1}}, \mathcal{F}_{2} \subseteq
{ V \brack k_{2}}, \ldots, \mathcal{F}_{r} 
\subseteq { V \brack k_{r}}$ are \emph{$r$-cross $t$-intersecting} if $\dim\left(F_{1} \cap F_{2} \cap \cdots \cap F_{r}\right) \geq t$ for all $F_{i} \in \mathcal{F}_{i}, i=1,2,\ldots,r$. When $r=2$, we simply say that $\mathcal F_1$ and $\mathcal F_2$ are cross $t$-intersecting.

The Erd\H{o}s--Ko--Rado theorem for vector spaces determines the maximum-sized $t$-intersecting families in ${V\brack k}$. More precisely, the theorem states that, for $n \geq 2k\geq2t$, every $t$-intersecting family has size at most ${n-t\brack k-t}$. Moreover, if $\mathcal{F}$ is $t$-intersecting and attains the bound, then either $\mathcal{F}$ consists of all $k$-subspaces containing a fixed $t$-subspace, or $n=2k$ and $\mathcal{F}$ consists of all $k$-subspaces contained in a fixed  $(2k-t)$-subspace. It was proved by Hsieh \cite{Hsieh-1975-1} in the 1970s for $n\geq2k+1$ and $(n,q)\neq(2k+1,2)$. The problem for the remaining parameter values turned out to be much more challenging, and the upper bound valid for all $n\geq2k$ was proved by Frankl and Wilson \cite{Frankl-Wilson-1986} via an ingenious linear algebra proof. They stated that for $n=2k$ it appeared likely that there were exactly two types of extremal families, and this was eventually verified by Tanaka \cite{Tanaka-2006} using deep machinery from the theory of association schemes.

The problem of investigating $r$-cross $t$-intersecting families with a large product of sizes has been extensively explored, yielding several insightful results and methods. Let $k_1\geq k_2\geq\cdots\geq k_r\geq t$, and let  $\mathcal{F}_{1}\subseteq{ V \brack k_{1}}$, $\mathcal{F}_{2}\subseteq{ V \brack k_{2}},\ldots,\mathcal{F}_{r}\subseteq{ V \brack k_{r}}$ be $r$-cross $t$-intersecting. Tokushige \cite{Tokushige-2013} determined the families maximizing $|\mathcal{F}_1||\mathcal{F}_2|$ for $r=2$ and $k_1=k_2$ via the eigenvalue method. Suda and Tanaka \cite{Suda-Tanaka} introduced a method based on semidefinite programming, and settled the problem for all $k_1$ and $k_2$ when $r=2, t=1$ and $n\geq\max\{2k_1,2k_2\}$. The method was then applied to several extremal problems (see, e.g., \cite{Suda-Tanaka-Tokushige,Tanaka-Tokushige}). For general $r\geq2$, Cao, Lu, Lv and Wang \cite{Cao-Lu-Lv-Wang} determined the families with maximum product of sizes for all $t\geq1$ and $n\geq k_1+k_2+t+1$. Wen and Lv \cite{Wen-Lv-vector} subsequently reduced the required lower bound to $n\geq k_1+k_r+1$ and $(n,q)\neq(k_1+k_r+1,2)$.

It has long been a challenging problem to establish intersection theorems when $n$ is at most $c$ ($c\in(1,2)$) times the largest dimension of the subspaces under  consideration. A central open problem in this direction is the following conjecture proposed by Frankl and Tokushige \cite{Frankl-2016}. We note that the case $r=2$ follows from the theorem of Suda and Tanaka \cite{Suda-Tanaka} mentioned above.
\begin{conjecture}[\cite{Frankl-2016}] \label{conj:1}
Let $n\geq\frac{r}{r-1}\cdot\max\{k_1,\ldots,k_r\}$. 
	If $\mathcal{F}_{1} \subseteq
	{ V \brack k_{1}},\ldots, 
	\mathcal{F}_{r} \subseteq { V \brack k_{r}}$ are $r$-cross $1$-intersecting, 
	then $\prod_{i=1}^{r}\left|\mathcal{F}_{i}\right| \leq \prod_{i=1}^{r}
	{ n-1 \brack k_{i}-1}$.
\end{conjecture}
\subsection{Main results}
In this paper, we overcome the barrier by proving the following  Erd\H{o}s--Ko--Rado type theorem for $r$-cross $t$-intersecting families. By setting $t=1$, we prove Conjecture \ref{conj:1} for all but at most three values of $n$. 
\begin{theorem} \label{theorem:1}
	Let $r\geq 3$, $k_{1}\geq k_{2}\geq\dots \geq k_{r} \geq t$ and $n \geq \frac{r}{r-1}\cdot(k_{1}+t)+1$. If $\mathcal{F}_{1} \subseteq { V \brack k_{1}}, \mathcal{F}_{2} \subseteq { V \brack k_{2}},\dots,  \mathcal{F}_{r} \subseteq { V \brack k_{r}}$ are $r$-cross $t$-intersecting, then
	$$
	\prod_{i =1}^{r}|\mathcal{F}_{i}| \leqslant \prod_{i =1}^{r}{ n-t \brack k_{i}-t}.
	$$
Equality holds if and only if $\mathcal{F}_{i}=\{ F \in {V \brack k_{i}} : T \subseteq F\}$ {\rm ($i=1,2,\ldots,r$)} for some $T \in {V \brack t}$.
\end{theorem}
We note that the coefficient $\frac{r}{r-1}$ is best possible for all $t\geq1$. To see this, let $k_1=\cdots=k_r=k$ and suppose that $k<n\leq\frac{rk-t}{r-1}$. For arbitrary $F_1,F_2,\ldots,F_r\in{V\brack k}$, we have
\begin{equation*}
	\dim(F_1\cap F_2\cap\cdots\cap F_r)\geq n-\sum_{i=1}^r\bigl(n-\dim(F_i)\bigr)=rk-(r-1)n\geq t.
\end{equation*}
Hence, $\mathcal{F}_1=\mathcal{F}_2=\cdots=\mathcal{F}_r={V\brack k}$ 
are $r$-cross $t$-intersecting.
\begin{remark}
During the final preparation of this paper, we became aware of an independent work of Shimizu and Tokushige \cite{Shimizu-Tokushige}. They prove Conjecture \ref{conj:1} when $k_1=\cdots=k_r$ and $k_1$ is sufficiently large depending on $q$ and $r$. One of the main tools in their proof is a junta theorem concerning intersecting linear maps obtained by Ellis, Kindler, and Lifshitz \cite{Ellis-Kindler-Lifshitz}. Their method also yields results for $r$-cross union families. 
\end{remark}
Our approach also applies to $r$-wise $t$-intersecting families and non-trivial $r$-cross $t$-intersecting families. A family $\mathcal{F}\subseteq{V\brack k}$ is \emph{$r$-wise $t$-intersecting} if $\dim(F_1\cap\cdots\cap F_r)\geq t$ for every choice of $F_1,\ldots,F_r\in\mathcal{F}$. This also serves as a natural generalization of the classical notion of $t$-intersecting. In 2010, Chowdhury and Patk\'{o}s \cite{Chowdhury} determined the maximum-sized $r$-wise $1$-intersecting families for $n\geq\frac{r}{r-1}\cdot k$, which is best possible. Our next main result settles the problem for general $t\geq2$ and all but at most $t+1$ values of $n$.
\begin{theorem} \label{theorem:3}
	Let $r \geq 3$, $k \geq t$ and $n\geq\frac{rk-t}{r-1}+t+1$. If $\mathcal{F} \subseteq { V \brack k}$ is a  $r$-wise $t$-intersecting family, then
	$$
	|\mathcal{F}| \leqslant { n-t \brack k-t}.
	$$
Equality holds if and only if $\mathcal{F}=\{ F \in {V \brack k} : T \subseteq F\}$ for some $T \in {V \brack t}$.
\end{theorem}

Our next main result is a Hilton--Milner (see \cite{HM}) type theorem for vector spaces. We focus on the problem for $r$-cross $t$-intersecting families, and refer the reader to \cite{AB,Cao-Lv-Wang-Zhou,D'haeseleer,Ihringer-2019} for works on $t$-intersecting families. Let $\mathcal{F}_{1} \subseteq { V \brack k_{1}}, \mathcal{F}_{2} \subseteq { V \brack k_{2}},\dots,  \mathcal{F}_{r} \subseteq { V \brack k_{r}}$ be $r$-cross $t$-intersecting. The $r$-tuple $(\mathcal{F}_1,\ldots,\mathcal{F}_r)$ is \emph{non-trivial} if no $t$-subspace is contained in all members of $\cup_i\mathcal{F}_i$. When there is no likelihood of confusion, we say that $\mathcal{F}_1,\ldots,\mathcal{F}_r$ are non-trivial. Our next main result determines for all $r\geq3$ the structure of non-trivial $r$-cross $t$-intersecting families maximizing the product of their sizes. 
\begin{theorem}\label{theorem:4}
	Let  $r \geq 3$, $k_{1}\geq k_{2}\geq\dots \geq k_{r} \geq t+1$ and $n \geq\frac{r}{r-1}\cdot(k_{1}+t)+4$. 
	If $\mathcal{F}_{1} \subseteq { V \brack k_{1}}, \mathcal{F}_{2} \subseteq { V \brack k_{2}},\dots,  \mathcal{F}_{r} \subseteq { V \brack k_{r}}$ are non-trivial $r$-cross $t$-intersecting families, then
	$$
	\prod_{i=1}^{r}\left|\mathcal{F}_{i}\right| \leq\left(q^{k_{r}-t}{t+1 \brack 1}{n-t-1 \brack k_{r}-t}+{n-t-1 \brack k_{r}-t-1}\right) \prod_{i=1}^{r-1}{n-t-1 \brack k_{i}-t-1}.
	$$
	Equality holds if and only if there exist $T \in {V \brack t+1}$ and $a \in[r]$ with $k_{a}=k_{r}$ such that $\mathcal{F}_{a}=\left\{F \in{V \brack k_{a}}: \dim(T \cap F) \geq t\right\}$ and $\mathcal{F}_{i}=\left\{F \in {V \brack k_{i}}: T \subseteq F\right\}\;(i \in[r] \setminus\{a\})$.
\end{theorem}
We note that the problem for $r=2$ is settled in \cite{Cao-Lu-Lv-Wang}, and our result substantially improves the bound $n\geq k_1+k_{r}+2$ given in \cite{Wen-Lv-vector}. This can also be regarded as a stability result for Theorem \ref{theorem:1}. More precisely, if $\prod_{i=1}^{r}|\mathcal{F}_{i}|$ is greater than the upper bound given in Theorem \ref{theorem:4}, then the families are trivial. We also note that we may assume $k_r\geq t+1$. Indeed, suppose that these families are non-trivial and $k_r=t$. Then every member of $\mathcal{F}_r$ is contained in each member of $\cup_{i\neq r}\mathcal{F}_i$. Set $M=\sum_{X\in\mathcal{F}_r}X$. Then $M\subseteq F$ whenever $F\in\cup_{i\neq r}\mathcal{F}_i$, and then $\mathcal{F}_r\subseteq{M\brack t}$ and $\mathcal{F}_i\subseteq\left\{F\in{V\brack k_i}:M\subseteq F\right\}$ for each $i\in[r-1]$. 
\subsection{Our techniques}
Our approach is purely combinatorial and is based on the notion of a \emph{$t$-cover}. The technique is useful for intersection problems on a variety of objects (see, e.g., \cite{Cao-Lv-Wang-Zhou,Cao-Lu-Lv-Wang,Wen-Lv-vector,Wen-Lv-2026}). The key step in applying the method is to find a suitable collection of $t$-covers (see Lemma \ref{lemmakey} and Proposition \ref{proposition:2.4}). Such subspaces can be used both to bound the relevant families and to characterize their structure. In order to prove our main results, the following refinements are essential. 
\begin{itemize}
\item\emph{Joint product estimates from partial intersections}.\; We consider two collections of families $\{\mathcal{G}_i:i\in[r]\}$ and $\{\mathcal{G}_{i,j}:i\neq j\in[r]\}$ (see (\ref{equfamgi}) and (\ref{equfamgij}) for details). These families accurately capture the property \emph{$r$-cross} rather than merely describing that the $r$ families are \emph{pairwise} cross $t$-intersecting. Then we measure how large an intersection a member of $\mathcal{F}_i$ can guarantee with every member of $\mathcal{G}_{i,j}$. If this parameter $m_{i,j}$ (see (\ref{equmij})) is large, the same choice provides a $t$-cover for $\mathcal{F}_j$ and, at the same time, allows us to estimate jointly the product of the sizes of the remaining $r-2$ families. Applying the corresponding estimate in the reverse direction controls the remaining family.

\item\emph{Truncated families of $t$-covers.}\;
To settle the complementary case in which all the $m_{i,j}$'s are less than a suitably chosen threshold $\nu$, we retain only the low-dimensional members of each $\mathcal{G}_i$ (see (\ref{equgip})). Then every $t$-cover of the truncated part of $\mathcal{G}_j$ is also a $t$-cover of $\mathcal{F}_i$ whenever $i\neq j$. On the other hand, the truncated part of $\mathcal{G}_i$ is itself a collection of $t$-covers of $\mathcal{F}_i$. Consequently, we can estimate the families $\mathcal{F}_1,\ldots,\mathcal{F}_r$ using the $t$-covering numbers of truncated families of $\mathcal{G}_i$'s and $\mathcal{G}_{i,j}$'s, and they can be matched cyclically when the estimates for the $r$ families are multiplied. This reduces the product estimate to the corresponding `diagonal' $t$-cover estimates. The threshold $\nu$ mentioned above is approximately $t+\frac{k_1+t}{r-1}$, which can be much smaller than the quantity $k_r$ that appears in the bound $k_1+k_r+2$ given in \cite{Wen-Lv-vector}.

\item \emph{A decomposition via $t$-covers.}
We decompose a family $\mathcal{F}$ with respect to two subspaces $T_1$ and $T_2$, according to whether $F\cap T_1$ fails to contain, equals, or properly contains $T_1\cap T_2$ (see (\ref{equdecomp})). In our applications, \(T_1\) is a \(t\)-cover of a 
subfamily of \(\mathcal{F}\). The simple yet useful decomposition provides both quantitative estimates and structural information. More precisely, the third part can be bounded using our strengthened $t$-cover estimate (see Proposition \ref{proposition:2.4}), and the other two parts allow us to control the remaining families.
\end{itemize}
\subsection{Organization}
The rest of this paper is organized as follows. In Section \ref{section 2}, we collect several inequalities and prove a key lemma for bounding the sizes of families using $t$-covers (Lemma \ref{lemmakey}), followed by a strengthened estimate in Proposition \ref{proposition:2.4}. In Section \ref{section 3}, we establish product estimates for the two complementary cases in which the relevant partial intersections have large or small dimension (Lemmas \ref{lemma:large-m} and \ref{lemma:small-m}, respectively), and then prove Theorems \ref{theorem:1} and \ref{theorem:3}. In Section \ref{section 4}, we use the techniques developed in the preceding sections together with the decomposition via $t$-covers described above to prove Theorem \ref{theorem:4}.

\section{Preliminaries and $t$-cover estimates} \label{section 2}
For a family $\mathcal{F}$ of subspaces of $V$, 
a subspace $T$ of $V$ is called a \emph{$t$-cover} of $\mathcal{F}$ if 
$\dim(T \cap F) \geq t$ for all $F \in \mathcal{F}$. If such a subspace exists, then the \emph{$t$-covering number} $\tau_{t}(\mathcal{F})$ of $\mathcal{F}$ is the minimum 
dimension of such a $t$-cover. It is clear that a $t$-intersecting family $\mathcal{F}$ 
is trivial if and only if $\tau_{t}(\mathcal{F})=t$. For simplicity, let 
\begin{align*}
	\mathcal{T}_{t}(\mathcal{F}) = \left \{ T \in {V \brack \tau_{t}(\mathcal{F})}: T \text{ is a $t$-cover of } \mathcal{F} \right \}.
\end{align*}
For a family $\mathcal{F}$ of subspaces of $V$ and a subspace $S$ of $V$, write $$\mathcal{F}[S] = \{ F \in \mathcal{F} : S \subseteq F\}.$$

Let $W$ be an $(e+\ell)$-dimensional vector space over $\mathbb{F}_q$(where $\ell,e \ge 1$), and let $L$ be a fixed $\ell$-subspace of $W$. We say that an $s$-subspace $U$ is of \emph{type} $(s,h)$ if $\dim(U\cap L)=h$. Define $\mathcal{M}(s,h;e+\ell,e)$ as the set of all subspaces of $W$ of type $(s,h)$.

\begin{Lemma}[\cite{Wang-Guo-Li}]\label{lem:2-2}
$\mathcal{M}(s,h;e+\ell,e)$ is non-empty if and only if $0\leq h\leq \ell$ and $0\leq s-h\leq e.$ Moreover, if $\mathcal{M}(s,h;e+\ell,e)$ is non-empty, then
$$
|\mathcal{M}(s,h;e+\ell,e)|=q^{(s-h)(\ell-h)}{e\brack s-h}{\ell\brack h}.	
$$
\end{Lemma}
By the lemma above, if \(A\le C\le V\), \(\dim A=a\), \(\dim C=c\) and \(a\le b\le c\), then
\begin{align} \label{eq:1}
    \left |  \left \{  B \in { C \brack b} : A \subseteq B\right \}\right | = {c-a \brack b-a}.
\end{align}

The following formulas are frequently used throughout the paper.
\begin{Lemma}\label{lem:8.1}
For $i\leq s$, the following hold.
\begin{itemize}
\item[$\rm (i)$] ${s\brack i}={s-1\brack i-1}+q^i{s-1\brack i}$ and ${s\brack i}=\frac{q^s-1}{q^i-1}\cdot{s-1\brack i-1}$.
\item[$\rm (ii)$] $q^{s-i}<\frac{q^s-1}{q^i-1}<q^{s-i+1}$ and $q^{i-s-1}<\frac{q^i-1}{q^s-1}<q^{i-s}$ for $i<s$.
\item[$\rm (iii)$] $q^{i-s-1}{s\brack i}<{s-1\brack i-1}<q^{i-s}{s\brack i}$ for $i<s$.
\item[$\rm (iv)$] $q^{i(s-i)}\leq{s\brack i} \leq q^{i(s-i)+\min\{i,s-i\}}$, and both inequalities hold strictly for $0<i<s$.
\end{itemize}
\end{Lemma}

The following elementary lemma provides a useful relation among the dimensions of certain intersections of subspaces. 
\begin{Lemma}\label{lem:2.7}
    Let $r \geq 2$, and let $T_{1}, T_{2}, \dots, T_{r}$ be subspaces of $V$. Write $G = T_{1} \cap T_{2} \cap \dots \cap T_{r}$ and $G_{i} = \bigcap_{h \neq i} T_{h}$ for $1\leq i\leq r$. Then for each $i \in [r]$, we have 
    $$
    \sum_{j \neq i} \dim G_{j} \leq \dim T_{i} + (r-2) \cdot \dim G.
    $$
\end{Lemma}
\begin{proof}
When $r = 2$, we have $G_1 = T_2$ and $G_2 = T_1$, and the inequality holds trivially. Now we assume that $r \geq 3$. By symmetry, it suffices to prove the case $i = 1$. Applying the dimension formula repeatedly yields 
\begin{align*}
    \dim(G_2 + \dots + G_r)=\sum_{j=2}^{r} \dim G_j-\sum_{j=3}^{r} \dim\bigl(G_j \cap (G_2 + \dots + G_{j-1})\bigr). 
\end{align*}
Let $j \in [3, r]$. It is clear that $G \subseteq G_j \cap (G_2 + \dots + G_{j-1})$. Conversely, 
\begin{align*}
    G_j \cap (G_2 + \dots + G_{j-1}) \subseteq  G_j \cap T_{j} = G.
\end{align*}
Hence $\dim\bigl(G_j \cap (G_2 + \dots + G_{j-1})\bigr)= \dim G$, implying that
\[
\begin{aligned}
\sum_{j=2}^{r} \dim G_j
= \dim(G_2 + \dots + G_r) + (r-2) \cdot \dim G 
\leq \dim T_1 + (r-2) \cdot \dim G,
\end{aligned}
\]
where the last inequality follows because $G_2 + \dots + G_r\leq T_1$.
\end{proof}
\begin{Lemma} \label{lem:2.6}
	Let $n \geq k \geq t$ and $\mathcal{F} \subseteq { V \brack k}$. If $G \in {V \brack \ell}$ is a $t$-cover of $\mathcal{F}$, then $|\mathcal{F}| \leq {\ell \brack t}{n- t \brack k-t}$.
\end{Lemma}
\begin{proof}
	Since $G$ is a $t$-cover of $\mathcal{F}$, we have $\dim(F\cap G)\geq t$ for al $F\in\mathcal{F}$. Hence $\mathcal{F}=\bigcup_{H \in {G \brack t}} \mathcal{F}[H]$. 
	By (\ref{eq:1}), $|\mathcal{F}[H]| \le {n-t \brack k-t}$ for each such $H$. So the estimate follows from the union bound.
\end{proof}
Let $n \geq k \geq t$ and $\mathcal{F} \subseteq {V \brack k}$. Each $t$-cover of $\mathcal{F}$ partially characterizes $\mathcal{F}$, as it 
intersects each member in a subspace of dimension at least $t$. This provides the intuition to utilize $t$-covers in establishing upper bounds on the size of $\mathcal{F}$. 
\begin{Lemma}\label{lemmakey}
Let $n \geq k+\ell-t+1$ and $\mathcal{F} \subseteq { V \brack k}$. Let $\mathcal{G}$ be a collection of $t$-covers of $\mathcal{F}$, where each has dimension at most $\ell$. Then for all $W\leq V$ with $\dim(W)\geq t$, it holds that
\begin{equation*}\label{eq:Prop:2-6}
	|\mathcal{F}[W]| \leq {\ell-t+1 \brack 1}^{\tau_{t}(\mathcal{G})-\dim(W)}{n- \tau_{t}(\mathcal{G}) \brack k-\tau_{t}(\mathcal{G})}.
\end{equation*}
\end{Lemma}
\begin{proof}
We adapt and refine the argument in \cite[Proposition 2.3]{Wen-Lv-vector}. Fix a subspace $W$ with $\dim(W)\geq t$. First, if $\dim(W)\geq\tau_{t}(\mathcal{G})$, then certainly
	\begin{align*}
		|\mathcal{F}[W]| \leq {n- \dim(W) \brack k-\dim(W)} 
		&\leq q^{-(n-k)(\dim(W)-\tau_{t}(\mathcal{G}))}{n- \tau_{t}(\mathcal{G}) \brack k-\tau_{t}(\mathcal{G})},
	\end{align*}
	and we have done as ${\ell-t+1 \brack 1}\leq q^{n-k}$. It remains to suppose  $\dim(W)<\tau_t(\mathcal{G})$. Now $W$ is not a $t$-cover of $\mathcal{G}$. We iteratively pick a sequence $W_0\subsetneqq W_1\subsetneqq\cdots\subsetneqq W_a$ of subspaces as follows. Put $W_0=W$. As long as $\dim(W_i)<\tau_t(\mathcal{G})$, the subspace $W_i$ is not a $t$-cover of $\mathcal{G}$. Hence, there exists $G_i\in\mathcal{G}$ such that $s_i:=\dim(W_i\cap G_i)<t$. Since $G_i$ is a $t$-cover of $\mathcal{F}$, every member of $\mathcal{F}[W_i]$ intersects $G_i+W_i$ in a subspace of dimension at least $\dim(W_i)+t-s_i$. Hence, there exists a $(\dim(W_i)+t-s_i)$-subspace $W_{i+1}$ with $W_i\leq W_{i+1}\leq G_i+W_i$ such that 
	$$
	|\mathcal{F}[W_i]|
	\leq{\dim(G_i)-s_i\brack t-s_i}|\mathcal{F}[W_{i+1}]|\leq{\ell-t+1\brack 1}^{t-s_i}|\mathcal{F}[W_{i+1}]|.
	$$
	Continue the procedure until the first index $a$ for which $\dim(W_a)\geq\tau_t(\mathcal{G})$. Since $\dim(W_{i+1})-\dim(W_i)=t-s_i$ for $i\leq a-1$, we obtain
	$$
	|\mathcal{F}[W_0]|\leq{\ell-t+1\brack 1}^{\dim(W_a)-\dim(W_0)}|\mathcal{F}[W_{a}]|\leq{\ell-t+1\brack 1}^{\dim(W_a)-\dim(W_0)}{n-\dim(W_a)\brack k-\dim(W_a)}.
	$$
	For $z\geq\tau_t(\mathcal{G})$, the function ${\ell-t+1\brack 1}^{z}{n-z\brack k-z}$ is decreasing, since $n \geq k+\ell-t+1$. Consequently, 
	$$|\mathcal{F}[W]|=|\mathcal{F}[W_0]|\leq{\ell-t+1\brack 1}^{\tau_t(\mathcal{G})-t}{n-\tau_t(\mathcal{G})\brack k-\tau_t(\mathcal{G})},$$
as desired.
\end{proof}
For simplicity, we set
\begin{align}
	f(s_{1},s_{2},a,b,y)&={s_{1} \brack y}{b-y+1 \brack 1}^{s_{2}-y}{n-s_{2} \brack a-s_{2}},\label{eq:f}\\
	\overline{f}(s_{1},s_{2},a,b)&={s_{1}-t \brack 1}{b-t+1 \brack 1}^{s_{2}-t-1}{n-s_{2} \brack a-s_{2}}\label{eq:b-f}.
\end{align}
\begin{proposition}\label{proposition:2.4}
    Let $n \geq k+\ell-t+1$ and $\mathcal{F} \subseteq { V \brack k}$. Suppose that $\mathcal{G}$  is a collection of $t$-covers of $\mathcal{F}$, where each has dimension at most $\ell$. Then the following hold.
  \begin{itemize}
  \item[\rm(i)]$|\mathcal{F}|\leq f(\tau_{t}(\mathcal{F}),\tau_{t}(\mathcal{G}),k,\ell,t)$.
  \item[\rm(ii)]Assume that $S\in{V\brack s}$ is a $(t+1)$-cover of $\mathcal{F}$.\label{eq:2.6.i.1} Then $|\mathcal{F}|\leq f(s,\tau_{t}(\mathcal{G}),k,\ell+1,t+1)$. Moreover, if every $F\in\mathcal{F}$ contains a common $t$-subspace of $S$, then $|\mathcal{F}|\leq\overline{f}(s,\tau_{t}(\mathcal{G}),k,\ell)$.
  \end{itemize}
\end{proposition}
\begin{proof}
Let $T$ be a $t$-cover of $\mathcal{F}$ of dimension $\tau_{t}(\mathcal{F})$. Then clearly $\mathcal{F}=\bigcup_{H \in {T \brack t}}\mathcal{F}[H]$. By picking $H_{1} \in {T \brack t }$ with $|\mathcal{F}[H]| \leq |\mathcal{F}[H_{1}]|$ whenever $H \in {T \brack t }$, we obtain that $|\mathcal{F}| \leq {\tau_{t}(\mathcal{F}) \brack t}|\mathcal{F}[H_{1}]|$. So (i) holds by applying Lemma \ref{lemmakey} to $W=H_1$. The proof of (ii) is almost the same, as we have $\mathcal{F}=\bigcup_{H \in {S \brack t+1 }}\mathcal{F}[H]$. In particular, suppose every $F\in\mathcal{F}$ contains $X\in{S\brack t}$, then $\mathcal{F}=\bigcup_{H}\mathcal{F}[H]$ with $H$ ranging over all $(t+1)$-subspaces of $S$ that contain $X$. There are ${s-t\brack 1}$ such subspaces, and so by the same argument as above, we derive $|\mathcal{F}|\leq\overline{f}(s,\tau_{t}(\mathcal{G}),k,\ell)$ from Lemma \ref{lemmakey} again.
\end{proof}
In the rest of this section, we collect several technical lemmas that will be used frequently in the subsequent sections.
\begin{Lemma}[\cite{Wen-Lv-vector}]\label{lem:8.4}
    Let $a,b\geq y$. If $n \geq a+b+1$ and $q \geq 3$, or $n \geq a+b+2$ and $q=2$, then the function $f(s,s, a, b,y)$ is decreasing as $s \in\{y, y+1, \ldots, a\}$ increases. In particular, we have $f(s,s, a, b,y) \leq {n-y \brack a-y}$, and equality holds precisely if $s=y$.
\end{Lemma}

\begin{Lemma} \label{lem:8.6}
    Let $ a,b\geq t+1$. If $n \geq a+b-t+2$ and $q \geq 3$, or $n \geq a+b-t+3$ and $q = 2$, then the function  $\overline{f}(s,s,a,b)$ is strictly decreasing as $s \in \{t+1,\dots,a\}$ increases. In particular, we have $\overline{f}(s,s,a,b) \leq {n-t-1 \brack a-t-1}$, and equality holds precisely if $s=t+1$. 
\end{Lemma}
\begin{proof}
For $x, y,z \geq 1$ with $x \leq y$, it is easy to check that
\begin{equation*}
    \frac{q^{x+1}-1}{q^{y+1}-1} \geq \frac{q^{x}-1}{q^{y}-1} \quad \text { and } \quad(q^{x}-1)(q^{y}-1)(q^{z}-1)<q^{x+y+z}-1.
\end{equation*}
For $s \in\{t+1, \ldots, a-1\}$, by Lemma \ref{lem:8.1} $\rm(i)$ and the inequalities above,
\begin{align}\label{eq:8.7.a}
\frac{\overline{f}(s,s, a, b)}{\overline{f}(s+1,s+1, a, b)}=& \frac{(q^{s-t}-1)(q-1)(q^{n-s}-1)}{(q^{s-t+1}-1)(q^{b-t+1}-1)(q^{a-s}-1)} \nonumber\\
\geq & \frac{(q-1)^{2}(q^{n-s}-1)}{(q^{2}-1)(q^{b-t+1}-1)(q^{a-s}-1)}> \frac{(q-1)^{2}(q^{n-s}-1)}{(q^{a+b-t+3-s}-1)}.
\end{align}
Then it is routine to check, by our assumption on $n$, that  $\overline{f}(s,s, a, b)>\overline{f}(s+1,s+1, a, b)$. Next, we have $\overline{f}(s,s,a,b)\leq\overline{f}(t+1,t+1, a, b)={n-t-1 \brack a-t-1}$, with equality if and only if $s=t+1$.
\end{proof}

Let us introduce the following expression.
\begin{align} \label{eq:4.2.*}
    g(k,\ell) = \sum_{i=t}^{\ell}q^{(k-i)(\ell-i)}{\ell \brack i}{n-\ell \brack k-i}.
\end{align}
Let $M$ be an $\ell$-subspace. It is routine to verify, based on Lemma \ref{lem:2-2}, that $g(k,\ell)$ equals the number of $k$-subspaces $F$ with $\dim(M \cap F) \geq t$. In particular,
\begin{align} \label{eq:4.2}
    g(k,t+1) = q^{k-t}{t+1 \brack 1}{n-t-1 \brack k-t}+{n-t-1 \brack k-t-1}.
\end{align}

\begin{Lemma}\label{lem:8.3}
    Let $k \geq t+1$ and $n \geq k+1$. Then $g(k,t+1)>q^{t}{n-t \brack k-t}> q^{n-k+t}{n-t-1 \brack k-t-1}$.
\end{Lemma}
\begin{proof}
By Lemma \ref{lem:8.1} $\rm(i)$, 
\begin{equation*}
    g(k, t+1)=\left(q^{t}+ \frac{(q^{t}-1)(q^{n-t}-q^{k-t+1}+q-1)}{(q-1)(q^{n-t}-1)}\right){
n-t  \brack
k-t
}.
\end{equation*}
So $g(k,t+1)>q^{t}{n-t \brack k-t}$. The second inequality follows directly from Lemma \ref{lem:8.1} $\rm(iii)$.
\end{proof}

\begin{Lemma}[\cite{Wen-Lv-vector}]\label{lem:8.8}
    If $n \geq k>s \geq t$, then
$$
q^{-(t+2)}g(k, s+1)<g(k, s)<g(k, s+1).
$$
\end{Lemma}

\begin{Lemma} \label{lem:8.9}
    Let  $r\ge 3$, $k_{1} \geq k_{2} \geq \dots \geq k_{r} \geq t+1$ and $ \ell \geq t+1$. For distinct $i,j \in [r]$, set
$$
h_{i,j}(s)=g(k_{j},s) \bar{f}(s,s,k_{i},\ell)^{r-1}.
$$
If $n \geq k_{1}+\ell-t+2+\frac{t+2}{r-1}$ and $q \geq 3$, or $n \geq k_{1}+\ell-t+3+\frac{t+2}{r-1}$ and $q = 2$, then $h_{i,j}(s)$ is strictly decreasing as $s \in \{t+1,\dots,k_i\}$ increases. In particular, we have $h_{i,j}(s) \leq g(k_{j},t+1) {n-t-1 \brack k_{i}-t-1}^{r-1}$, and equality holds precisely if $s=t+1$. 
\end{Lemma}
\begin{proof}
Fix $s \in\{t+1, \ldots, k_i-1\}$. By Lemma \ref{lem:8.8} and the same argument as we derive (\ref{eq:8.7.a}), 
\begin{align*}
\frac{h_{i,j}(s)}{h_{i,j}(s+1)}&> q^{-t-2}\left(\frac{(q-1)^{2}(q^{n-s}-1)}{(q^{k_{i}+\ell-t+3-s}-1)} \right)^{r-1}\\
&=\left(\frac{(q-1)^{2}(q^{n-s}-1)}{q^{\frac{t+2}{r-1}}(q^{k_{i}+\ell-t+3-s}-1)} \right)^{r-1}\geq \left(\frac{(q-1)^{2}(q^{n-s}-1)}{q^{k_{1}+\ell-t+3+\frac{t+2}{r-1}-s}-1} \right)^{r-1}.
\end{align*}
Then it is routine to check that $h_{i,j}(s)>h_{i,j}(s+1)$. Hence the lemma follows from $h_{i,j}(t+1)= g(k_{j},t+1) {n-t-1 \brack k_{i}-t-1}^{r-1}$.
\end{proof}
\begin{Lemma}[\cite{Wen-Lv-vector}] \label{lem:8.10}
    Let $r\ge 3$, $k_{1} \geq k_{2} \geq \dots \geq k_{r} \geq t+1$ and $n \geq k_{1}+t+1$. Set
$$
h_{2}(a)=g\left(k_{a}, t+1\right) \prod_{i \neq a}{n-t-1 \brack k_{i}-t-1},\;a=1,2,\ldots,r.
$$
For each $a, b \in [r]$ with $a \geq b$, we have $h_{2}(a) \geq h_{2}(b)$, and equality holds if and only if $k_{a} = k_{b}$.
\end{Lemma}

\section{Product bounds for $r$-cross $t$-intersecting families} \label{section 3}

In this section, we derive product estimates for \(r\)-cross \(t\)-intersecting families in two complementary cases, according to the dimensions of certain partial intersections introduced below. Then we prove Theorems \ref{theorem:1} and \ref{theorem:3}. 
\subsection{Joint product estimates from partial intersections} \label{section 3.1}
We begin with a simple yet useful lemma.
\begin{Lemma}\label{lem:10.1}
    Let $r\geq3$, $n > d_{1}\geq d_{2}\geq\dots\geq d_{r-1}$ and $\mathcal{A}_{i} \subseteq { V \brack d_{i}}$ {\rm(}$i=1,2,\ldots,r-1${\rm)}. Put 
    \begin{align*}
        d = \min \{ \dim(\cap_{y \in [r-1]}A_{y}): A_{y} \in \mathcal{A}_{y}, y \in [r-1] \}.
    \end{align*}
    Then for each $i \in [r-1]$, we have 
    \begin{align*}
        \prod_{h \neq i}|\mathcal{A}_{h}| <  q^{(d+1)(d_{1}-d)-(r-2)(n-d_{1})}\prod_{ h \neq i}{ n-d+1 \brack d_{h}-d+1 }.
    \end{align*}
\end{Lemma}
\begin{proof}
    Pick $A_{y} \in \mathcal{A}_{y}\;(y \in [r-1])$ such that $\dim( \cap_{y \in [r-1]  } A_{y} ) = d$. Let $h \in [r-1]\setminus \{ i\}$ and $a_{h} = \dim(\cap_{y \in [r-1] \setminus \{ h\}} A_{y})$. By the definition of $d$, we obtain that $\cap_{y \in [r-1] \setminus \{ h\}} A_{y} $ is a $d$-cover of $\mathcal{A}_{h}$. By Lemma \ref{lem:2.6}, we have  
\begin{align*}
    |\mathcal{A}_{h}| \leq {a_{h}\brack d}{n- d \brack d_{h}-d}.
\end{align*} 
    Therefore, by Lemma \ref{lem:8.1} $\rm(iii)$ and $\rm(iv)$, we have
\begin{align*}
    \prod_{ h \neq i }|\mathcal{A}_{h}| &\leq \prod_{ h \neq i}\left({a_{h}\brack d}{n- d \brack d_{h}-d}\right) 
    < \prod_{ h \neq i}\left(q^{(d+1)(a_{h}-d)-(n-d_{h}) }{n- d+1 \brack d_{h}-d+1}\right) \\
     &\leq q^{(d+1)(d_{1}-d)-(r-2)(n-d_{1})} \prod_{ h \neq i}{ n-d+1 \brack d_{h}-d+1 },
\end{align*}
where the last inequality follows from Lemma \ref{lem:2.7} and $d_{1}\geq d_{2}\geq\dots\geq d_{r-1}$.
\end{proof}
Let $r \geq 3$ and $k_{1}\geq k_{2} \geq \dots\geq k_{r} \geq t$. Assume that $\mathcal{F}_{1} \subseteq { V \brack k_{1}}, \mathcal{F}_{2} \subseteq { V \brack k_{2}}, \dots,  \mathcal{F}_{r} \subseteq { V \brack k_{r}}$ are $r$-cross $t$-intersecting families. Following \cite{Wen-Lv-vector}, set
\begin{equation}\label{equfamgi}
    \mathcal{G}_{i}=\{ \cap_{j \neq i} F_{j} : F_{j} \in \mathcal{F}_{j}, j \in [r] \setminus \{ i\} \},\quad i \in [r].
\end{equation}
For every \(F_i\in\mathcal{F}_i\) and \(G\in\mathcal{G}_i\), the \(r\)-cross \(t\)-intersection property gives $\dim(F_i\cap G)\geq t$. Thus every member of \(\mathcal{G}_i\) is a \(t\)-cover of \(\mathcal{F}_i\), and vice versa. In that paper, the \(t\)-covering numbers of the families \(\mathcal{G}_i\) were used to estimate the individual families \(\mathcal{F}_i\). Let
\begin{equation}\label{equfamgij}
\mathcal{G}_{i, j}=\{ \cap_{h \neq i, j} F_{h} : F_{h} \in \mathcal{F}_{h}, h \in [r] \setminus \{ i,j\}\},\quad i\neq j \in [r].
\end{equation}
For a fixed \(F\in\mathcal{F}_i\), the quantity
\[
\min_{G\in\mathcal{G}_{i,j}}\dim(F\cap G)
\]
is the largest integer \(p\) for which \(F\) is a \(p\)-cover of \(\mathcal{G}_{i,j}\). We therefore set
\begin{equation}\label{equmij}
    m_{i,j}= \max_{F \in \mathcal{F}_{i}}\min_{ G \in \mathcal{G}_{i,j}} \dim( F \cap G)
\end{equation}
for distinct $i,j\in[r]$, and write
\begin{equation}\label{equdefm}
    m = \max_{i\neq j}m_{i,j}.
\end{equation}

The following lemma shows that an $m_{i,j}$-cover of $\mathcal{G}_{i,j}$ 
lying in $\mathcal{F}_i$ can be used to estimate the product of the sizes of the $(r-1)$ families other than $\mathcal{F}_i$.
\begin{Lemma} \label{lemma:joint-product}
    Let $n>k_{1}$. Assume that $\mathcal{F}_{1} \subseteq { V \brack k_{1}}, \mathcal{F}_{2} \subseteq { V \brack k_{2}},\dots,  \mathcal{F}_{r} \subseteq { V \brack k_{r}}$ are $r$-cross $t$-intersecting families. Let $i,j \in [r]$ be distinct. Then  
    \begin{equation*}
    \prod_{ h \neq i }|\mathcal{F}_{h}|
    <  q^{E} { n-t \brack k_{j}-t }\prod_{ h \neq i,j }{ n-m_{i,j}+1 \brack k_{h}-m_{i,j}+1 },
    \end{equation*}
   where $E=(m_{i,j}-t)t+\min\{t,m_{i,j}-t\}+(m_{i,j}+1)(k_{1}-m_{i,j})-(r-2)(n-k_{1})$.
\end{Lemma}
\begin{proof}
Pick $F_{i} \in \mathcal{F}_{i}$ and $G\in \mathcal{G}_{i,j}$ such that $$\dim( F_{i} \cap G ) =\min\limits_{ G_{i,j} \in \mathcal{G}_{i,j}} \dim( F_{i} \cap G_{i,j}) = m_{i,j}.$$ By the definition of $\mathcal{G}_{i,j}$ and the assumption that $\mathcal{F}_1,\ldots,\mathcal{F}_r$ are $r$-cross $t$-intersecting, we have $\dim(F_{j}\cap F_{i} \cap G) \geq t$ for all $F_{j} \in \mathcal{F}_{j}$, and so $F_{i} \cap G$ is a $t$-cover of $\mathcal{F}_{j}$. Then by Lemma \ref{lem:2.6},
\begin{align} \label{eq:3.2.3.1}
    |\mathcal{F}_{j}| \leq {m_{i,j} \brack t} {n-t \brack k_{j}-t} \leq q^{\,t(m_{i,j}-t)+\min\{t,\;m_{i,j}-t\}}{n-t \brack k_{j}-t},
\end{align}
where the second inequality follows from Lemma \ref{lem:8.1} $\rm(iv)$.

Next, we consider the $(r-1)$ families $\{ F_{i}\} \subseteq \mathcal{F}_{i}$ and $ \mathcal{F}_{h}\;(h \in [r] \setminus \{ i,j\})$. Observe that
\begin{align*}
    \min \{ \dim(F_{i}\cap(\cap_{h \in [r]\setminus\{i,j\}}F_{h})): F_{h} \in \mathcal{F}_{h}, h \in [r]\setminus\{i,j\} \}= \min_{ G_{i,j} \in \mathcal{G}_{i,j}} \dim( F_{i} \cap G_{i,j}) =m_{i,j}.
\end{align*}
By Lemma \ref{lem:10.1}, we have
\begin{align} \label{eq:3.2.3.3}
    \prod_{ h \neq i,j }|\mathcal{F}_{h}|
    <  q^{(m_{i,j}+1)(k_{1}-m_{i,j})-(r-2)(n-k_{1})} \prod_{ h \neq i,j }{ n-m_{i,j}+1 \brack k_{h}-m_{i,j}+1 }.
\end{align}
  Then the desired inequality follows from (\ref{eq:3.2.3.1}) and (\ref{eq:3.2.3.3}).
\end{proof}
The following lemma shows that if $m$ exceeds an explicit threshold, then the product of sizes of an $r$-tuple of $r$-cross $t$-intersecting families is small.
\begin{Lemma}\label{lemma:large-m}
	Let $\nu\geq t$ with $(r-1)(\nu-t+1)\geq k_1+t$, 
	and let $n\geq k_1+\nu+2$. Assume that
	$\mathcal{F}_1\subseteq{V\brack k_1},\ldots,
	\mathcal{F}_r\subseteq{V\brack k_r}$
	are $r$-cross $t$-intersecting families satisfying
	$m\geq\nu+1$. Then
	\[
	\prod_{h=1}^r|\mathcal{F}_h|
	<
	\prod_{h=1}^r{n-t\brack k_h-t}.
	\]
\end{Lemma}
\begin{proof}
Let $i,j \in [r] $ be distinct indices with $m_{i,j} = m$. By Lemma \ref{lemma:joint-product} and $
\min\{t,m-t\}\leq t$, 
\begin{align} \label{eq:3.2.2.3}
    \prod_{ h \neq i }|\mathcal{F}_{h}|
    &<  q^{(m-t+1)t+(m+1)(k_{1}-m)-(r-2)(n-k_{1})} { n-t \brack k_{j}-t }\prod_{ h \neq i,j}{ n-m+1 \brack k_{h}-m+1 }\nonumber \\
    &<  q^{(m-t+1)t+(m+1)(k_{1}-m)-(r-2)(m-t)(n-k_{1})}\prod_{ h \neq i}{ n-t \brack k_{h}-t },
\end{align} 
where the second inequality follows from Lemma \ref{lem:8.1} $\rm(iii)$.

On the other hand, since $m_{j,i} \leq m$, there exist $F_{j} \in \mathcal{F}_{j}, G \in \mathcal{G}_{j,i}$ such that $\dim( F_{j} \cap G ) \leq m$. Observe that $F_{j} \cap G$ is a $t$-cover of $\mathcal{F}_{i}$. By Lemmas \ref{lem:2.6} and \ref{lem:8.1} $\rm(iv)$, we have 
\begin{align} \label{eq:3.2.2.2}
    |\mathcal{F}_{i}| \leq {m \brack t} {n-t \brack k_{i}-t} \leq q^{(m-t+1)t}{n-t \brack k_{i}-t}.
\end{align}
where the second inequality follows from Lemma \ref{lem:8.1} $\rm (iv)$.
 
Combining (\ref{eq:3.2.2.3}) and (\ref{eq:3.2.2.2}), we obtain 
\begin{align*}
    \prod_{ h = 1}^{r}|\mathcal{F}_{h}|< q^{2t(m-t+1)+(m+1)(k_{1}-m)-(r-2)(m-t)(n-k_{1})} \prod_{ h = 1}^{r}{ n-t \brack k_{h}-t }.
\end{align*}
	It remains to show that the exponent of $q$ is negative.
By rewriting it as 
$$
(m-t)\left(k_1+t-m-1+
\frac{(t+1)(k_1-t)+2t}{m-t}
-(r-2)(n-k_1)\right),
$$
we see that it is strictly decreasing as $m\geq\nu+1$ increases. It therefore
suffices to consider $m=\nu+1$. Substituting $m=\nu+1$ into the original
exponent, we obtain
\begin{align*}
	2t(\nu-t+2)+(\nu+2)(k_1-\nu-1)
	-(r-2)(\nu-t+1)(n-k_1).
\end{align*}
Since $n\geq k_1+\nu+2$, this is at most
$$2t(\nu-t+2)
+(\nu+2)\bigl(k_1-t-(r-1)(\nu-t+1)\bigr)\leq2t(\nu-t+2)-2t(\nu+2)
=-2t^2<0,$$
where the first inequality follows from
$(r-1)(\nu-t+1)\geq k_1+t$. This completes the proof.
\end{proof}

\subsection{Truncated families of $t$-covers} \label{section 3.2}
We now consider the complementary case in which all the parameters \(m_{i,j}\) are small. For each $i\in [r]$ and $p \geq t$, we denote
\begin{equation}\label{equgip}
\mathcal{G}^{p}_{i}=\{ G \in \mathcal{G}_{i} :  \dim G \leq p\},
\end{equation}
and write 
$$\tau_{t,i}^{p}:=\tau_{t}(\mathcal{G}^{p}_{i}).$$ 
In particular, $\mathcal{G}^{p}_{i}$ is a collection of \(t\)-covers of \(\mathcal{F}_i\), all of dimension at most \(p\). Let us record several basic properties of these parameters.
\begin{Lemma} \label{lem:3-A}
With the notation above, let $i,j\in[r]$ be distinct and
suppose that $m_{i,j}\leq p$. Then the following statements hold.
    \begin{itemize}
	\item [$\rm (i)$] for each $F_{i} \in \mathcal{F}_{i}$, there exists $G \in \mathcal{G}_{i,j}$ such that $\dim(F_{i} \cap G) \leq p$; in particular, $\mathcal{G}_{j}^{p} \neq \varnothing$;
	\item [$\rm (ii)$] $\tau_{t}(\mathcal{F}_{i}) \leq p$ and $\tau_{t,i}^{p} \leq k_{i}$;
	\item [$\rm (iii)$] $\dim(T_{j} \cap F_{i} ) \geq t$ for each $T_{j} \in \mathcal{T}_{t}(\mathcal{G}_{j}^{p})$ and $F_{i} \in \mathcal{F}_{i}$;
	\item [$\rm (iv)$] $\tau_{t}(\mathcal{F}_{i}) \leq \tau_{t,j}^{p}$.
\end{itemize}
\end{Lemma}
\begin{proof}
$\rm (i)$ Let $F_{i} \in \mathcal{F}_{i} $. By the definition of $m_{i,j}$ and $m$, there exists $G \in \mathcal{G}_{i,j}$ such that $$\dim(F_{i} \cap G) \leq m_{i,j}\leq p.$$ In particular, we have $F_{i}\cap G \in \mathcal{G}_{j}^{p}$ from $F_{i}\cap G \in \mathcal{G}_{j}$. Thus $\mathcal{G}_{j}^{p}$ is non-empty.
 
 $\rm (ii)$ For each $G \in \mathcal{G}_{i}^{p}$ and $F_i \in \mathcal{F}_i$, the $r$-cross $t$-intersecting property gives $\dim(F_i \cap G) \ge t$. Thus $\mathcal{F}_i$ and $\mathcal{G}_{i}^{p}$ are cross $t$-intersecting. Take any $G \in \mathcal{G}_{i}^{p}$, we obtain 
 $\tau_{t}(\mathcal{F}_{i}) \leq \dim G \leq p$. Similarly, take any $F_i \in \mathcal{F}_i$, it follows that 
 $\tau_{t,i}^{p} = \tau_t(\mathcal{G}_i^p) \leq \dim F_i = k_i$.
 
 $\rm (iii)$ Let $T_{j}$ be a $t$-cover of $\mathcal{G}^{p}_{j}$. For each $ F_{i} \in \mathcal{F}_{i}$, by $\rm (i)$, there exists $G \in \mathcal{G}_{i,j}$ such that $F_{i} \cap G \in \mathcal{G}^{p}_{j}$, implying that $\dim(T_{j}\cap F_{i}) \geq \dim(T_{j} \cap F_{i} \cap G) \geq t$. 
 
 $\rm (iv)$ By the definition of $t$-cover and $\rm (iii)$, we have $\tau_{t}(\mathcal{F}_{i}) \leq \tau_{t,j}^{p}$.
\end{proof}

\begin{Lemma} \label{lemma:small-m}
 Let $\nu\geq t$ and $n\geq k_1+\nu+2$. Suppose that
 $\mathcal{F}_1\subseteq{V\brack k_1},\ldots,
 \mathcal{F}_r\subseteq{V\brack k_r}$
 are $r$-cross $t$-intersecting families and satisfy
 $m\leq\nu$, then
    \begin{equation*}
    \prod_{ h =1}^{r}|\mathcal{F}_{h}| \leq \prod_{ h =1}^{r}{ n-t \brack k_{h}-t },
    \end{equation*}
equality holds if and only if $\mathcal{F}_{h}=\{ F \in {V \brack k_{h}} : T \subseteq F\}\;(h =1,2,\dots,r)$ for some $T \in {V \brack t}$.
\end{Lemma}
\begin{proof}
Let 
$$(h,\ell) \in \{(1,2),(2,3),\dots,(r-1,r),(r,1)\}.$$
Note that $\mathcal{G}^{\nu}_{h}$ is a collection of $t$-covers of $\mathcal{F}_{h}$. Then by Lemma \ref{lem:3-A} $\rm (iv)$ and Proposition \ref{proposition:2.4} $\rm (i)$, 
\begin{align*}
    |\mathcal{F}_{h}| \leq f(\tau_{t,\ell}^{\nu},\tau_{t,h}^{\nu},k_{h},\nu,t).
\end{align*}
    Therefore, by (\ref{eq:f}), we have
\begin{align*}
    \prod_{h=1}^{r}|\mathcal{F}_{h}| \leq \prod_{h=1}^{r}f(\tau_{t,h+1}^{\nu},\tau_{t,h}^{\nu},k_{h},\nu,t)=\prod_{h=1}^{r}f(\tau_{t,h}^{\nu},\tau_{t,h}^{\nu},k_{h},\nu,t),
\end{align*}
where $\tau_{t,r+1}^{\nu} = \tau_{t,1}^{\nu}$. This together with Lemma \ref{lem:8.4}  yields 
\begin{align*} 
    \prod_{ h =1}^{r}|\mathcal{F}_{h}| \leq \prod_{ h =1}^{r}{ n-t \brack k_{h}-t },
\end{align*}
and equality can hold only if $\tau_{t,1}^{\nu}=\tau_{t,2}^{\nu}=\dots =\tau_{t,r}^{\nu} =t$.

It remains to characterize equality. Suppose equality holds. Let $h,\ell \in [r]$ be distinct. By Lemma \ref{lem:3-A} $\rm (iii)$, there exists
$T_{\ell} \in \mathcal{T}_{t}(\mathcal{G}_{\ell}^{\nu})$ such that $
    \mathcal{F}_{h} \subseteq \left\{ F \in {V \brack k_{h}} : T_{\ell} \subseteq F\right\}$. Therefore, 
\begin{align*}
    \left( \prod_{h=1}^{r}|\mathcal{F}_{h}|\right)^{r-1}  \leq \prod_{\ell =1}^{r}\prod_{h\neq \ell}\left|\left\{ F \in {V \brack k_{h}} : T_{\ell} \subseteq F\right\}\right|= \left(\prod_{h=1}^{r}{ n-t \brack k_{h}-t} \right)^{r-1},
\end{align*}
which implies that $
    \mathcal{F}_{h}= \left\{ F \in {V \brack k_{h}} : T_{\ell} \subseteq F\right\}$ for any two distinct $ h,\ell \in [r] $. Then we get  $T_{1}=T_{2}=\dots=T_{r}$ from $r \geq 3$, which finishes the proof.
\end{proof}

\subsection{Proofs of Theorems \ref{theorem:1} and \ref{theorem:3}} \label{section 3.3}
\noindent{\bf Proof of Theorem \ref{theorem:1}.}\;
Put
$$
\nu:=\left\lceil\frac{k_1+t}{r-1}\right\rceil+t-1.
$$
Since $n$ is an integer, the assumption
$n\geq\frac{r}{r-1}(k_1+t)+1$ implies
$$
n\geq
k_1+t+1+\left\lceil\frac{k_1+t}{r-1}\right\rceil
=k_1+\nu+2.
$$
Moreover, $
(r-1)(\nu-t+1)\geq k_1+t$. If $m\geq\nu+1$, then Lemma \ref{lemma:large-m} gives the strict inequality $
\prod_{h=1}^r|\mathcal{F}_h|<\prod_{h=1}^r{n-t\brack k_h-t}$. If $m\leq\nu$, then Lemma \ref{lemma:small-m} gives $
\prod_{h=1}^r|\mathcal{F}_h|\leq\prod_{h=1}^r{n-t\brack k_h-t}$, 
with equality if and only if there exists $T\in{V\brack t}$ with
$$
\mathcal{F}_h=
\left\{F\in{V\brack k_h}:T\subseteq F\right\},\;h=1,2,\ldots,r.
$$
This completes the proof.
\hfill$\square$ \vspace{1em}

\noindent {\bf Proof of Theorem \ref{theorem:3}.}\;If \(k=t\), the result is immediate. Hence assume that \(k\ge t+1\). Let $ p = \left \lceil \frac{k-t}{r-1} \right\rceil+t $, 
$$\mathcal{G}=\{\cap_{i=1}^{r-2}F_{i}: F_{1}, F_{2}, \dots, F_{r-2} \in \mathcal{F}\},$$ 
and
$$
\mathcal{G}^\prime=\{F \cap G: F\in\mathcal{F},\;G \in \mathcal{G},\;\dim(F \cap G) \leq p-1\}.$$ It is obvious that $p \geq t+1$. By the assumption on $n$, we have $ n \geq k+p+1$. We divide
the proof into the following two cases. 

\noindent\textbf{Case 1.} There exists $F_{0} \in \mathcal{F}$ such that $F_{0} \cap G \notin \mathcal{G}^\prime$ for each  $G \in \mathcal{G}$. 

Let $m_{0} =\min \{ \dim(F_{0}\cap G): G \in \mathcal{G}\}$, and pick $F_{1}^\prime, F_{2}^\prime, \dots, F_{r-2}^\prime \in \mathcal{F}$ such that $$\dim(F_{0} \cap (\cap_{i=1}^{r-2}F_{i}^\prime)) = m_{0}.$$ Note that the hypothesis of this case gives \(m_0 \ge p\). By Lemma \ref{lem:2.7} and the pigeonhole principle, there exists $j \in [r-2]$ such that 
$$\dim(F_{0} \cap (\cap_{i \in [r-2] \setminus \{ j\}}F_{i}^\prime)) \leq \frac{k+(r-3)m_{0}}{r-2}.$$ By the definitions of $m_{0}$ and $\mathcal{G}$, we have 
$$\dim(F_{0}\cap\left(F \cap (\cap_{i \in [r-2] \setminus \{ j\}}F_{i}^\prime)\right)) \geq m_{0}$$ for all $F \in \mathcal{F}$, and so the subspace $F_{0} \cap (\cap_{i \in [r-2] \setminus \{ j\}}F_{i}^\prime)$ is an $m_{0}$-cover of $\mathcal{F}$. By Lemmas \ref{lem:2.6} and \ref{lem:8.1} $\rm(iii)$-$\rm(iv)$, we have
\begin{align*}
    |\mathcal{F}| &\leq {\lfloor \frac{k+(r-3)m_{0}}{r-2}\rfloor \brack m_{0}}{n-m_{0} \brack k-m_{0}}< q^{(m_{0}+1)\left  \lfloor \frac{k-m_{0}}{r-2}\right  \rfloor-(m_{0}-t)(n-k)} {n-t \brack k-t}.
\end{align*}
To prove the theorem, we only need to show that
\begin{align*}
    Q:=(m_{0}+1) \left  \lfloor \frac{k-m_{0}}{r-2}\right  \rfloor-(n-k)(m_{0}-t)\leq 0.
\end{align*}

Since $m_{0} \geq p \geq \frac{k-t}{r-1}+t$ and $n \geq \frac{rk+(r-2)t+(r-1)}{r-1}$, it is easy to check that $m_{0}-t \geq \left  \lfloor \frac{k-m_{0}}{r-2}\right  \rfloor $ and $n \geq k+t+1+\left  \lfloor \frac{k-m_{0}}{r-2}\right  \rfloor $. Hence
\begin{align*}
    Q
    &=(t+1)\left  \lfloor \frac{k-m_{0}}{r-2}\right  \rfloor-(m_{0}-t)\left(n-k-\left  \lfloor \frac{k-m_{0}}{r-2}\right  \rfloor\right)\\
    &\leq -(m_{0}-t)\left(n-k-t-1-\left  \lfloor \frac{k-m_{0}}{r-2}\right  \rfloor\right) \leq 0.
\end{align*}

\noindent\textbf{Case 2.} For each $F \in \mathcal{F}$, there exists $G \in \mathcal{G}$ such that $F \cap G \in \mathcal{G}^\prime$. 

Let $T$ be a $t$-cover of $\mathcal{G}^\prime$. For each $F \in \mathcal{F}$, there exists $G \in \mathcal{G}$ such that $F \cap G \in \mathcal{G}^\prime$, implying that $\dim(T \cap F) \geq \dim(T\cap F \cap G) \geq t$. Therefore, by the definition of $t$-cover, we have $\tau_t(\mathcal{F}) \leq \tau_t(\mathcal{G}^\prime)$. Observe that $\mathcal{G}^\prime$ is a collection of $t$-covers of $\mathcal{F}$. Then we derive from Proposition \ref{proposition:2.4} $\rm (i)$ that 
\begin{align*}
    |\mathcal{F}| \leq f(\tau_t(\mathcal{F}),\tau_t(\mathcal{G}^\prime),k,p-1,t) \leq f(\tau_t(\mathcal{G}^\prime),\tau_t(\mathcal{G}^\prime),k,p-1,t) \leq {n-t \brack k-t},
\end{align*}
where the second and the third inequalities follow from $\tau_t(\mathcal{F}) \leq \tau_t(\mathcal{G}^\prime)$ and Lemma \ref{lem:8.4}, respectively. Moreover, $|\mathcal{F}|={n-t \brack k-t}$ equality holds if and only if $\tau_t(\mathcal{G}^\prime) = t$, which is equivalent to $\mathcal{F}$ being trivial. \qed

\section{Non-trivial $r$-cross $t$-intersecting families} \label{section 4}
In this section, we use the techniques developed in the preceding sections to prove Theorem \ref{theorem:4}. For convenience, throughout this section, we suppose $r \geq 3$ and $k_{1} \geq k_{2} \geq \dots \geq k_{r} \geq t+1$, and put 
\begin{equation}\label{equdefnu'}
    \nu^\prime :=  \left \lceil \frac{k_{1}+t}{r-1}\right \rceil+t.
\end{equation}
Let $\mathcal{F}_{1} \subseteq { V \brack k_{1}}, \mathcal{F}_{2} \subseteq { V \brack k_{2}},\dots,  \mathcal{F}_{r} \subseteq { V \brack k_{r}}$ be $r$-cross $t$-intersecting. The families are \emph{maximal} if for any $r$-cross $t$-intersecting families $\mathcal{H}_{1} \subseteq { V \brack k_{1}}, \mathcal{H}_{2} \subseteq { V \brack k_{2}},\dots,  \mathcal{H}_{r} \subseteq { V \brack k_{r}}$, we have $\mathcal{F}_i\subseteq\mathcal{H}_i,\;i=1,2,\ldots,r$ implies $\mathcal{F}_i=\mathcal{H}_i,\;i=1,2,\ldots,r$.

\begin{Lemma}\label{lem:10.10}
    Let $n\geq k_1+\nu'+4$. Assume that $\mathcal{F}_{1} \subseteq { V \brack k_{1}}, \mathcal{F}_{2} \subseteq { V \brack k_{2}}, \dots,  \mathcal{F}_{r} \subseteq { V \brack k_{r}}$ are $r$-cross $t$-intersecting families with $m \geq  \nu^\prime+1$. Then 
    \begin{align*}
    \prod_{ h = 1}^{r}|\mathcal{F}_{h}| 
    <& g(k_{r},t+1)\prod_{ h = 1}^{r-1}{ n-t-1 \brack k_{h}-t-1 }.
\end{align*}
\end{Lemma}
\begin{proof}
Let $i,j \in [r] $ be two distinct indices with $m_{i,j} = m$. By Lemma \ref{lemma:joint-product}, we have 
\begin{align} \label{eq:4.2.2.3}
    \prod_{ h \neq i }|\mathcal{F}_{h}|
    &<  q^{(m-t)t+\min\{t,m-t\}+(m+1)(k_{1}-m)-(r-2)(n-k_{1})} { n-t \brack k_{j}-t }\prod_{ h \neq i,j}{ n-m+1 \brack k_{h}-m+1 } \nonumber \\
    &<  q^{n-k_{j}+1+(m-t)t+\min\{t,m-t\}+(m+1)(k_{1}-m)-(r-2)(m-t-1)(n-k_{1})}\prod_{ h \neq i}{ n-t-1 \brack k_{h}-t-1 } \nonumber \\
    &\leq  q^{n-t+(m-t)t+\min\{t,m-t\}+(m+1)(k_{1}-m)-(r-2)(m-t-1)(n-k_{1})}\prod_{ h \neq i}{ n-t-1 \brack k_{h}-t-1 },
\end{align}
where in the second and the third steps we used Lemma \ref{lem:8.1} $(\rm iii)$ and $k_{j} \geq t+1$, respectively.

On the other hand, since $m_{j,i} \leq m$, there exist $F_{j} \in \mathcal{F}_{j}$ and $G \in \mathcal{G}_{j,i}$ such that $\dim( F_{j} \cap G ) \leq m$. Observe that $F_{j} \cap G$ is a $t$-cover of $\mathcal{F}_{i}$. By Lemmas \ref{lem:2.6}, \ref{lem:8.1} $(\rm iv)$ and \ref{lem:8.3}, we have 
\begin{align} \label{eq:4.2.2.2}
    |\mathcal{F}_{i}| \leq {m \brack t} {n-t \brack k_{i}-t} \leq q^{(m-t)t+\min\{t,m-t\}}{n-t \brack k_{i}-t}<q^{-t+(m-t)t+\min\{t,m-t\}}g(k_{i},t+1).
\end{align}
 
Combining (\ref{eq:4.2.2.3}) and (\ref{eq:4.2.2.2}) and applying Lemma \ref{lem:8.10}, we obtain
\begin{align*} 
    \prod_{ h = 1}^{r}|\mathcal{F}_{h}| 
    < q^{E(m,n)}\cdot g(k_{r},t+1)\prod_{ h = 1}^{r-1}{ n-t-1 \brack k_{h}-t-1 },
\end{align*}
where 
\begin{align*}
E(m,n):=&n-2t+2t(m-t)+2\min\{t,m-t\}\\
&+(m+1)(k_1-m)
 -(r-2)(m-t-1)(n-k_1).
\end{align*}
To prove the lemma, we only need to show that $E(m,n)\leq0$. If $\nu'\geq k_1$, then the condition $m\geq\nu'+1$ is
impossible. Hence we may assume that $\nu'\leq k_1-1$.
By the definition of $\nu'$, we have $
(r-1)(\nu'-t)\geq k_1+t$. Moreover, since $k_1\geq t+1$, we have $
(r-2)(\nu'-t)
\geq\frac{r-2}{r-1}(k_1+t)>1$. Thus $E(m,n)$ is decreasing in $n$. Next, it is routine to check that $E(m,n)$ is decreasing on $m\in\{\nu'+1,\ldots, k_1\}$. Consequently, it suffices to take \(m=\nu'+1\) and \(n-k_1=\nu'+4\). Substituting these values and collecting terms, we obtain
\[
\begin{aligned}
E(m,n)\leq{}&
(\nu'+3)\bigl(k_1-t-(r-1)(\nu'-t)\bigr)
+2t(\nu'-t)\\
&+2\min\{t,\nu'-t+1\}+t+2
-(r-3)(\nu'-t).
\end{aligned}
\]
By the definition of \(\nu'\) in (\ref{equdefnu'}), we obtain $
k_1-t-(r-1)(\nu'-t)\leq-2t$. This together with \(\min\{t,\nu'-t+1\}\leq t\) yields 
\[
\begin{aligned}
E(m,n)
&\leq-2t(\nu'+3)+2t(\nu'-t)+3t+2
-(r-3)(\nu'-t)\\
&=-2t^2-3t+2-(r-3)(\nu'-t)<0,
\end{aligned}
\]
where the last inequality follows from \(r\geq3\) and \(t\geq1\).
\end{proof}
\begin{Lemma}\label{lem:10.17}
    Let $n\geq k_1+\nu'+4$. Assume that $\mathcal{F}_{1} \subseteq { V \brack k_{1}}, \mathcal{F}_{2} \subseteq { V \brack k_{2}}, \dots,  \mathcal{F}_{r} \subseteq { V \brack k_{r}}$ are $r$-cross $t$-intersecting families satisfying $m \leq  \nu^\prime$. If there exist two distinct $ i,j \in [r]$ such that
\begin{align} \label{eq: t-3-1}
    \prod_{h \neq i,j }|\mathcal{F}_{h}| &\leq q^{r+1+(\nu^\prime+2)(k_{1}-\nu^\prime-1)-(r-2)(\nu^\prime-t)(n-k_{1})} \prod_{h \neq i,j }{n-t-1 \brack k_{h}-t-1},
\end{align}
then 
    \begin{align*}
    \prod_{ h = 1}^{r}|\mathcal{F}_{h}| 
    <& g(k_{r},t+1)\prod_{ h = 1}^{r-1}{ n-t-1 \brack k_{h}-t-1 }.
\end{align*}
\end{Lemma}
\begin{proof}
    Observe that $\mathcal{G}^{\nu^\prime}_{i}$ is a collection of $t$-covers of $\mathcal{F}_{i}$, and $\mathcal{G}^{\nu^\prime}_{j}$ is a collection of $t$-covers of $\mathcal{F}_{j}$. By Lemma \ref{lem:3-A} $\rm (iv)$, Proposition \ref{proposition:2.4} $\rm (i)$ and (\ref{eq:f}), we have
\begin{align*}
|\mathcal{F}_{i}||\mathcal{F}_{j}| &\leq f(\tau_{t,j}^{\nu^\prime},\tau_{t,i}^{\nu^\prime},k_{i},\nu^\prime,t)f(\tau_{t,i}^{\nu^\prime},\tau_{t,j}^{\nu^\prime},k_{j},\nu^\prime,t)
=f(\tau_{t,i}^{\nu^\prime},\tau_{t,i}^{\nu^\prime},k_{i},\nu^\prime,t)f(\tau_{t,j}^{\nu^\prime},\tau_{t,j}^{\nu^\prime},k_{j},\nu^\prime,t).
\end{align*}
It follows from Lemma \ref{lem:8.4} that
\begin{align*}
    |\mathcal{F}_{i}||\mathcal{F}_{j}| \leq { n-t \brack k_{i}-t}{ n-t \brack k_{j}-t}.
\end{align*}
By combining this with (\ref{eq: t-3-1}), we obtain that
\begin{align*}
    \prod_{ h =1}^{r}|\mathcal{F}_{h}| \leq q^{r+1+(\nu^\prime+2)(k_{1}-\nu^\prime-1)-(r-2)(\nu^\prime-t)(n-k_{1})} { n-t \brack k_{i}-t}{ n-t \brack k_{j}-t}\prod_{h \neq i,j }{n-t-1 \brack k_{h}-t-1}.
\end{align*}
This together with Lemmas \ref{lem:8.1} $(\rm iii)$ and \ref{lem:8.3} yields
\begin{align*}
    \prod_{ h =1}^{r}|\mathcal{F}_{h}| &< q^{r+1-t+n-k_{j}+1+(\nu^\prime+2)(k_{1}-\nu^\prime-1)-(r-2)(\nu^\prime-t)(n-k_{1})} g(k_{i},t+1)\prod_{h \neq i }{n-t-1 \brack k_{h}-t-1}\\
    &\leq q^{r+1-2t+n+(\nu^\prime+2)(k_{1}-\nu^\prime-1)-(r-2)(\nu^\prime-t)(n-k_{1})} g(k_{r},t+1)\prod_{h =1}^{r-1}{n-t-1 \brack k_{h}-t-1},
\end{align*}
where the second inequality follows from $k_{j} \geq t+1$ and Lemma \ref{lem:8.10}. 

To prove the lemma, we need only to verify
\begin{align*}
    Q:=r+1-2t+n+(\nu^\prime+2)(k_{1}-\nu^\prime-1)-(r-2)(\nu^\prime-t)(n-k_{1})\leq 0.
\end{align*}
By the definition of $\nu'$, we have $(r-1)(\nu'-t)\geq k_1+t$, and so $(r-2)(\nu'-t)>1$, and hence $Q$ is
decreasing in $n$. Since $n-k_1\geq\nu'+4$, it follows that
\begin{align*}
Q
&\leq(\nu'+3)
 \big(k_1-t-(r-1)(\nu'-t)\big)\\
&\quad+(\nu'-t)+r+2-t
 -\big((r-2)(\nu'-t)-1\big)\\
&\leq-2t(\nu'+3)+r+3-t-(r-3)(\nu'-t)\leq-2t(\nu'+3)+6-t<0.
\end{align*}
This completes the proof.
\end{proof}

Let $\mathcal{F}_{1} \subseteq { V \brack k_{1}}, \mathcal{F}_{2} \subseteq { V \brack k_{2}}, \dots,  \mathcal{F}_{r} \subseteq { V \brack k_{r}}$ be $r$-cross $t$-intersecting families. We now turn to the situation in which $\mathcal{F}_{1},\mathcal{F}_{2}, \dots,\mathcal{F}_{r}$ fail to satisfy the hypotheses of Lemmas $\ref{lem:10.10}$ and $\ref{lem:10.17}$. For simplicity, we say that $\mathcal{F}_{1},\mathcal{F}_{2}, \dots,\mathcal{F}_{r}$ satisfy \emph{property $\mathcal{P}$} if $m \leq \nu^\prime$ and  
\begin{equation}\label{equpropertyp}
\prod_{h \neq i,j }|\mathcal{F}_{h}| > q^{r+1} \cdot q^{(\nu^\prime+2)(k_{1}-\nu^\prime-1)-(r-2)(\nu^\prime-t)(n-k_{1})} \prod_{h \neq i,j }{n-t-1 \brack k_{h}-t-1}
\end{equation}
for each pair of distinct $i,j \in [r]$. 

For a family $\mathcal{F} \subseteq {V \brack k} $ and two subspaces $ T_{1},T_{2} $ of $ V$, we have the following decomposition.
\begin{equation}\label{equdecomp}
	\mathcal{F}=\mathcal{F}^{<}(T_{1}, T_{2})\cup\mathcal{F}^{=}(T_{1}, T_{2})\cup\mathcal{F}^{>}(T_{1}, T_{2}),
\end{equation}
where
    \begin{align*}
        \mathcal{F}^{<}(T_{1}, T_{2})&=\left\{ F \in \mathcal{F}: T_{1}\cap T_{2} \nsubseteq F\cap T_{1}\right\};\\
        \mathcal{F}^{=}(T_{1}, T_{2})&=\left\{ F \in \mathcal{F}:T_{1} \cap T_{2}=F \cap T_{1}\right\};\\
        \mathcal{F}^{>}(T_{1}, T_{2})&=\left\{ F \in \mathcal{F}:T_{1} \cap T_{2} \subsetneqq F \cap T_{1}\right\}.
    \end{align*}

\begin{Lemma} \label{lem:10.2}
    Let $n \geq k_{1}+\nu^\prime-t+1$. Assume that $\mathcal{F}_{1} \subseteq { V \brack k_{1}}, \mathcal{F}_{2} \subseteq { V \brack k_{2}}, \dots,  \mathcal{F}_{r} \subseteq { V \brack k_{r}}$ are non-trivial $r$-cross $t$-intersecting families satisfying property $\mathcal{P}$. Suppose that there exists $z \in [r]$ with $\tau_{t,z}^{\nu^\prime} \geq t+1$. Let $T \in \mathcal{T}_{t}(\mathcal{G}_{z}^{\nu^\prime})$ and let $F$ be a subspace of $V$ such that $\dim(T\cap F)=t$. For each $i \in [r] \setminus \{ z\}$, the following hold:
    \begin{itemize}
    \item [$\rm (i)$] $|\mathcal{F}_{i}^{>}(T, F)| \leq \bar{f}(\tau_{t,z}^{\nu^\prime},\tau_{t,i}^{\nu^\prime},k_{i},\nu^\prime)$;
    \item [$\rm (ii)$] if there exist $j \in [r] \setminus \{ z, i\}$ and $ F_{j} \in \mathcal{F}_{j}$ with $T \cap F = T \cap F_{j} $, then $|\mathcal{F}_{i}^{<}(T,F)| < \frac{1}{q^{r+1}}|\mathcal{F}_{i}|$;
    \item [$\rm (iii)$] if there exist $j \in [r] \setminus \{ z,i\}$ and $F_{j} \in \mathcal{F}_{j}$ with $T \cap F \nsubseteq T \cap F_{j} $, then $|\mathcal{F}_{i}^{=}(T,F)| < \frac{1}{q^{r+1}}|\mathcal{F}_{i}|$.
\end{itemize}
\end{Lemma}

\begin{proof}

    $\rm (i)$ Observe that $\mathcal{G}_{i}^{\nu^\prime}$ is a collection of $t$-covers of $\mathcal{F}_{i}$, $T$ is a $(t+1)$-cover of $\mathcal{F}_{i}^{>}(T, F)$, and $T\cap F$ is a $t$-cover of $\mathcal{F}_{i}^{>}(T, F)$. It follows from Proposition \ref{proposition:2.4} $\rm (ii)$ that $(\rm i)$ holds.

    $\rm (ii)$ We only need to consider the case that $\mathcal{F}_{i}^{<}(T,F) \neq \varnothing$. Note that $F_{j} \in \mathcal{F}_{j}^{=}(T,F)$. Then $\mathcal{F}_{j}^{=}(T,F)\neq \varnothing$. Consider the families $\mathcal{F}_{i}^{<}(T,F),\;\mathcal{F}_{j}^{=}(T,F)$ and $\mathcal{F}_{h}\;(h \in [r] \setminus \{ z,i,j\})$. 
    Let 
    \begin{align*}
        m^\prime = \min \{ \dim(\cap_{h \neq z}F_{h}^\prime): F_{i}^\prime \in  \mathcal{F}_{i}^{<}(T,F),\;F_{j}^\prime \in\mathcal{F}_{j}^{=}(T,F),\;F_{h}^\prime \in  \mathcal{F}_{h}\;(h \in [r] \setminus \{ z,i,j\}) \}.
    \end{align*} 
     For each $F_{i}^\prime \in \mathcal{F}_{i}^{<}(T,F),\;F_{j}^\prime \in \mathcal{F}_{j}^{=}(T,F)$ and $F_{h}^\prime \in \mathcal{F}_{h}\;(h \in [r] \setminus \{ z,j,i\})$, since $$\dim(T\cap(\cap_{h \in [r] \setminus \{ z\}}F_{h}^\prime))\leq \dim(T\cap F_{i}^\prime)<t,$$ 
    we have $\cap_{h \in [r] \setminus \{ z\}}F_{h}^\prime \notin \mathcal{G}_{z}^{\nu^\prime}$, and so $\dim(\cap_{h \in [r] \setminus \{ z\}}F_{h}^\prime) \geq \nu^\prime+1$ from $\cap_{h \in [r] \setminus \{ z\}}F_{h}^\prime \in \mathcal{G}_{z}$. Therefore, $m^\prime \geq \nu^\prime+1$.
    
    By Lemma \ref{lem:10.1}, we have
    \begin{align*} 
    |\mathcal{F}_{i}^{<}(T,F)|\prod_{h \neq z,i,j}|\mathcal{F}_{h}| &< q^{(m^\prime+1)(k_{1}-m^\prime)-(r-2)(n-k_{1})} \prod_{h \neq z,j }{n-m^\prime+1 \brack k_{h}-m^\prime+1} \\
    &\leq q^{(m^\prime+1)(k_{1}-m^\prime)-(r-2)(m^\prime-t-1)(n-k_{1})} \prod_{h \neq z,j }{n-t-1 \brack k_{h}-t-1},
    \end{align*}
    where the second inequality follows from Lemma \ref{lem:8.1} $(\rm iii)$.
    It is clear that the function 
\[
Q(m') = (m'+1)(k_{1}-m')-(r-2)(m'-t-1)(n-k_{1})
\]
is decreasing for \(m' \in \bigl\{ \nu^\prime+1,\nu^\prime+2, \dots, k_{1} \bigr\}\), and so
\begin{align*}
    |\mathcal{F}_{i}^{<}(T,F)|\prod_{h \neq z,i,j}|\mathcal{F}_{h}|<q^{(\nu^\prime+2)(k_{1}-\nu^\prime-1)-(r-2)(\nu^\prime-t)(n-k_{1})} \prod_{h \neq z,j }{n-t-1 \brack k_{h}-t-1}.
\end{align*}
Since $\mathcal{F}_{1},\mathcal{F}_{2}, \dots,\mathcal{F}_{r}$ satisfy property $\mathcal{P}$, we have
    \begin{align*}
        |\mathcal{F}_{i}^{<}(T,F)|\prod_{h \neq z,i,j}|\mathcal{F}_{h}|<  \frac{1}{q^{r+1}} \prod_{h \neq z,j}|\mathcal{F}_{h}|,
    \end{align*}
implying $\rm (ii)$.
    
    $\rm (iii)$ We only need to consider the case that $\mathcal{F}_{i}^{=}(T,F) \neq \varnothing$. Observe that $F_{j} \in \mathcal{F}_{j}^{<}(T,F)$. Thus $\mathcal{F}_{j}^{<}(T,F)\neq \varnothing$. Applying the same argument as in the proof of $\rm (ii)$ to the $(r-1)$ families $\mathcal{F}_{i}^{=}(T,F),\;\mathcal{F}_{j}^{<}(T,F)$ and $\mathcal{F}_{h}\;(h \in [r] \setminus \{ z,j,i\})$ yields
    \begin{align*}
        |\mathcal{F}_{i}^{=}(T,F)|\prod_{h \neq z,j,i}|\mathcal{F}_{h}| <  \frac{1}{q^{r+1}} \prod_{h \neq z,j}|\mathcal{F}_{h}|,
    \end{align*}
    which immediately implies $(\rm iii)$.
\end{proof}

\begin{Lemma} \label{lem:10.5}
Let $n \geq k_{1}+t+1$. Let $\mathcal{F}_{1} \subseteq { V \brack k_{1}}, \mathcal{F}_{2} \subseteq { V \brack k_{2}}, \dots,  \mathcal{F}_{r} \subseteq { V \brack k_{r}}$ be maximal non-trivial $r$-cross $t$-intersecting families with $m \leq \nu^\prime$. Then there is at most one $h \in [r]$ for which $\tau_{t,h}^{\nu^\prime}=t$.
\end{Lemma}
\begin{proof}
    Suppose to the contrary that there exist two distinct $ i, j  \in [r]$ such that $\tau_{t,i}^{\nu^\prime}  = \tau_{t,j}^{\nu^\prime}  =t$. Pick $T_{i} \in \mathcal{T}_{t}(\mathcal{G}_{i}^{\nu^\prime}) $ and $T_{j} \in \mathcal{T}_{t}(\mathcal{G}_{j}^{\nu^\prime})$. By Lemma \ref{lem:3-A} $\rm (iii)$, we have $T_{i} \subseteq F $ for each $F \in \cup_{h \neq i}\mathcal{F}_{h}$, and $T_{j} \subseteq F_{i}$ for each $F_{i} \in \mathcal{F}_{i}$. Pick $H \in { V \brack k_{i}-t}$ such that $\dim(H\cap (T_{i}+T_{j})) =0$. Since $\mathcal{F}_{1}, \mathcal{F}_{2},\dots, \mathcal{F}_{r}$ are maximal, we have $H+T_{i} \in \mathcal{F}_{i}$, and so $T_{j} \subseteq H+T_{i}$. Then
\begin{align*}
    t=\dim((H+T_{i})\cap T_{j})&= \dim(H+T_{i})+\dim(T_{j})-\dim(H+T_{i}+ T_{j})\\
    &= k_{i}+t-\dim(H)-\dim(T_{i}+T_{j})
    =\dim(T_{i}\cap T_{j}),
\end{align*}
implying that $T_{i} = T_{j}$. Thus $T_{i} \subseteq F $ for each $F \in \cup_{h =1}^{r}\mathcal{F}_{h}$, contradicting the non-triviality of $\mathcal{F}_{1}, \mathcal{F}_{2},\dots, \mathcal{F}_{r}$.
\end{proof} 

\begin{Lemma} \label{lem:10.11}
Let $n \geq k_{1}+t+1$. Assume that $\mathcal{F}_{1} \subseteq { V \brack k_{1}}, \mathcal{F}_{2} \subseteq { V \brack k_{2}}, \dots,  \mathcal{F}_{r} \subseteq { V \brack k_{r}}$ are maximal non-trivial $r$-cross $t$-intersecting families satisfying $m \leq \nu^\prime$. Suppose that $z \in [r]$ satisfies $\tau_{t,z}^{\nu^\prime}= t$, and $T_{z}\in \mathcal{T}_{t}(\mathcal{G}^{\nu^\prime}_{z})$. For each $i \in [r] \setminus \{ z\}$, there exists $T_{i} \in \mathcal{T}_{t}(\mathcal{G}^{\nu^\prime}_{i})$ such that $T_{z} \subseteq T_{i}$.
\end{Lemma}

\begin{proof}
Suppose for contradiction that for some $i \in [r] \setminus \{z\}$, we have $\dim(T_i \cap T_z) < t$ for each $T_i \in \mathcal{T}_t(\mathcal{G}_i^{\nu'})$. Choose $T \in \mathcal{T}_{t}(\mathcal{G}^{\nu^\prime}_{i})$ that maximizes $\dim(T \cap T_z)$. We pick $H_{1} \in {T_{z} \brack 1} \setminus {T \brack 1} $ and $ H_{2} \in {T \brack \dim(T)-1}$ with $T \cap T_{z} \subseteq H_{2}$. Then $\dim (H_{1} + H_{2}) = \dim T = \tau_{t,i}^{\nu^\prime}$, and 
\begin{align*}
    \dim(T_{z} \cap (H_{1} + H_{2}) ) \geq 
     \dim(T_{z} \cap H_{2} )+\dim(T_{z} \cap H_{1})
    \geq \dim(T_{z} \cap T )+1>\dim(T_{z} \cap T ).
\end{align*}
By the choice of $T$, we have $H_{1} + H_{2} \notin \mathcal{T}_{t}(\mathcal{G}^{\nu^\prime}_{i})$. Hence there exist $F_{z} \in \mathcal{F}_{z} $ and $ G \in \mathcal{G}_{z,i}$ such that $F_{z} \cap G \in \mathcal{G}^{\nu^\prime}_{i}$ and $\dim((H_{1} + H_{2}) \cap F_{z} \cap G)\leq t-1$. It follows that
\begin{align*}
    \dim(T\cap F_{z} \cap G) \leq \dim(H_{2}\cap F_{z} \cap G)+1\leq \dim((H_{1}+H_{2})\cap F_{z} \cap G)+1\leq t.
\end{align*}
Therefore, we deduce $\dim(T \cap F_{z} \cap G)=t$ from $T \in \mathcal{T}_{t}(\mathcal{G}^{\nu^\prime}_{i})$.

Now we claim that $H_{1} \subseteq (T\cap G)+F_{z}$. Let $H$ be a complement of $T \cap F_{z}\cap G$ in $F_{z}$. From $T_{z}\in \mathcal{T}_{t}(\mathcal{G}^{\nu^\prime}_{z})$ and Lemma \ref{lem:3-A} $\rm (iii)$, we have $T_{z} \subseteq F^\prime$ for each $F^\prime \in \cup_{h \in [r] \setminus \{ z\}}\mathcal{F}_{h}$. Then by the maximality of $\mathcal{F}_{1}, \mathcal{F}_{2},\dots, \mathcal{F}_{r}$, the family $\mathcal{F}_z$ contains all $k_z$-subspaces containing $T_z$, and so there exists $F \in  \mathcal{F}_{z}$ with $H + T_{z}  \subseteq F$. Hence 
\begin{align*}
    \dim(F\cap G)\leq t+\dim(H\cap G) = \dim(T \cap F_{z}\cap G)+\dim(H\cap G) \leq\dim(F_{z}\cap G)\leq \nu^\prime,
\end{align*}
implying that $F \cap G \in \mathcal{G}^{\nu^\prime}_{i}$. Therefore,
\begin{align*}
     \dim(F \cap ((T\cap G)+F_{z}) )&\geq \dim((F \cap T\cap G)+(F\cap F_{z})) \\     
     &= \dim(F\cap T\cap G)+\dim(F\cap F_{z})-\dim(F\cap T \cap F_{z}\cap G)\\
     &\geq t+\dim(F\cap H) \geq k_{z},
\end{align*}
which yields $F \subseteq (T\cap G)+F_{z}$. So our claim follows from $H_{1} \subseteq T_{z} \subseteq F$.

By the fact that $T_{z}\subseteq F^\prime$ for each $F^\prime \in \cup_{h \in [r] \setminus \{ z\}}\mathcal{F}_{h}$, it holds that $H_{1} \subseteq T_{z} \subseteq G$. Hence, using  $(T\cap G)+H_{1}\subseteq(H_{1}+T)\cap G$, we obtain that
\begin{align*}
    \dim(((T\cap G)+H_{1})\cap F_{z}) \leq \dim((H_{1}+T)\cap G \cap F_{z} ) \leq \dim((H_{1} + H_{2}) \cap F_{z} \cap G)+1 \leq t.
\end{align*}
However, this leads to
\begin{align*}
    \dim((T\cap G)+H_{1}+F_{z})&=\dim(F_{z})+\dim((T\cap G)+H_{1})-\dim(((T\cap G)+H_{1})\cap F_{z})\\
    &\geq\dim(F_{z})+\dim(T\cap G)+1-t
    =\dim((T\cap G)+F_{z})+1,
\end{align*}
which contradicts that $H_{1} \subseteq (T\cap G)+F_{z}$.
\end{proof}

\begin{Lemma} \label{lem:8.2}
    Let $r \geq 1$. Then we have
    \begin{align*}
        \frac{q^{r+1}}{q^{r+1}-2} \leq q^{\frac{1}{r}}.
    \end{align*}
\end{Lemma}
\begin{proof}
By Bernoulli's inequality, together with $q^{r-1} \geq r$ and $q^{2}-2 \geq q$, we obtain
\begin{align*}
    \left(1-\frac{2}{q^{r+1}}\right)^{r} \geq 1-\frac{2r}{q^{r+1}} \geq 1- \frac{2}{q^{2}} \geq \frac{1}{q}.
\end{align*}
So the desired inequality holds.
\end{proof}

\begin{Lemma} \label{lem:10.6}
  Let $n \geq\max\{k_{1}+t+1,k_{1}+\nu^\prime-t+3+\frac{t+2}{r-1}\}$. Assume that $\mathcal{F}_{1} \subseteq { V \brack k_{1}}, \mathcal{F}_{2} \subseteq { V \brack k_{2}}, \dots,  \mathcal{F}_{r} \subseteq { V \brack k_{r}}$ are maximal non-trivial $r$-cross $t$-intersecting families satisfying property $\mathcal{P}$. If there exists $z \in [r]$ such that $\tau_{t,z}^{\nu^\prime}= t$, then 
\begin{align} \label{eq:10.6.A}
    \prod_{h=1}^{r}|\mathcal{F}_{h}|
    \leq g(k_{r},t+1)\prod_{h=1}^{r-1}{n- t-1 \brack k_{h}-t-1},
\end{align}
and equality holds if and only if there exists $T \in {V \brack t+1}$ and $a \in[r]$ with $k_{a}=k_{r}$ such that $\mathcal{F}_{a}=\left\{F \in{V \brack k_{a}}: \dim(T \cap F) \geq t\right\}$ and $\mathcal{F}_{h}=\left\{F \in {V \brack k_{h}}: T \subseteq F\right\}\;(h \in[r] \setminus\{a\})$.

\end{Lemma}
\begin{proof}
 From $\tau_{t,z}^{\nu^\prime} = t$ and Lemma \ref{lem:10.5}, we have $\tau_{t,\ell}^{\nu^\prime} \geq t+1$ for each $\ell \in [r] \setminus \{z \}$. Pick $T_{z}\in \mathcal{T}_{t}(\mathcal{G}^{\nu^\prime}_{z})$. Next, by Lemma \ref{lem:10.11}, we can pick $T_{\ell} \in \mathcal{T}_{t}(\mathcal{G}^{\nu^\prime}_{\ell})$ such that $T_{z} \subseteq T_{\ell}$ for each $\ell \in [r] \setminus \{z \}$. From $T_{z}\in \mathcal{T}_{t}(\mathcal{G}^{\nu^\prime}_{z})$ and Lemma \ref{lem:3-A} $\rm (iii)$, we have $T_{z} \subseteq F^\prime$ for each $F^\prime \in \cup_{h \in [r] \setminus \{ z\}}\mathcal{F}_{h}$. The proof is then divided into two cases.

\noindent\textbf{Case 1.} There exist two distinct $i, j \in [r] \setminus \{ z\} $ and $F_{j} \in \mathcal{F}_{j} $ such that $T_{z} =F_{j} \cap T_{i} $. 

By Lemma \ref{lem:10.2} $\rm (ii)$, we have 
\begin{align*}
        |\mathcal{F}_{z}|=|\mathcal{F}_{z}^{<}(T_{i},F_{j})|+(|\mathcal{F}_{z}^{=}(T_{i},F_{j})|+|\mathcal{F}_{z}^{>}(T_{i},F_{j})|) < \frac{1}{q^{r+1}}|\mathcal{F}_{z}|+{n-t \brack k_{z}-t},
    \end{align*}
implying that
\begin{align}\label{eq:5.9.1}
    |\mathcal{F}_{z}|< \frac{q^{r+1}}{q^{r+1}-1}{n-t \brack k_{z}-t} \leq q^{\frac{1}{r}}{n-t \brack k_{z}-t},
\end{align}
where the second inequality follows from Lemma \ref{lem:8.2}.

On the other hand, let $h,\ell \in [r] \setminus \{ z\}$ be distinct. For each $F \in \mathcal{F}_{h}$, observe that $T_{\ell} \cap T_{z} = T_{z} \subseteq F\cap T_{\ell}$. Hence $\mathcal{F}_{h} =\mathcal{F}_{h}^{=}(T_{\ell},T_{z}) \cup \mathcal{F}_{h}^{>}(T_{\ell},T_{z})$. By the facts that $T_{z}\subseteq F^\prime$ for each $F^\prime \in \cup_{h \in [r] \setminus \{ z\}}\mathcal{F}_{h}$ and that $\mathcal{F}_{1}, \mathcal{F}_{2}, \dots,  \mathcal{F}_{r} $ are non-trivial,  it holds that there exists $F_{0} \in \mathcal{F}_{z} $ with $T_{\ell} \cap T_{z} = T_{z} \nsubseteq F_{0}\cap T_{\ell}$. Then Lemma \ref{lem:10.2} $\rm (iii)$ yields
\begin{align*}
    |\mathcal{F}_{h}^{=}(T_{\ell},T_{z})|<\frac{1}{q^{r+1}}|\mathcal{F}_{h}|,
\end{align*}
and so we derive
\begin{align*}
   |\mathcal{F}_{h}|
   < \frac{q^{r+1}}{q^{r+1}-1}|\mathcal{F}_{h}^{>}(T_{\ell},T_{z})| \leq q^{\frac{1}{r}}\bar{f}(\tau_{t,\ell}^{\nu^\prime},\tau_{t,h}^{\nu^\prime},k_{h},\nu^\prime),
\end{align*}
where the second inequality follows from Lemmas \ref{lem:8.2} and \ref{lem:10.2} $\rm (i)$. Therefore, by (\ref{eq:b-f}), we have
\begin{align*} 
    \prod_{h \neq z}|\mathcal{F}_{h}| \leq \left(\prod_{\ell \neq z}q^{\frac{r-2}{r}}\prod_{h \neq z,\ell}\bar{f}(\tau_{t,\ell}^{\nu^\prime},\tau_{t,h}^{\nu^\prime},k_{h},\nu^\prime)\right)^{\frac{1}{r-2}}
    =q^{\frac{r-1}{r}}\prod_{h \neq z}\bar{f}(\tau_{t,h}^{\nu^\prime},\tau_{t,h}^{\nu^\prime},k_{h},\nu^\prime).
\end{align*}
It follows from Lemma \ref{lem:8.6} that
\begin{align*}
    \prod_{h \neq z}|\mathcal{F}_{h}| \leq q^{\frac{r-1}{r}}\prod_{h\neq z}{n- t-1 \brack k_{h}-t-1}.
\end{align*}
By combining this with (\ref{eq:5.9.1}), we obtain that
\begin{align*}
    \prod_{h =1}^{r}|\mathcal{F}_{h}| < q{n-t \brack k_{z}-t}\prod_{h\neq z}{n- t-1 \brack k_{h}-t-1}.
\end{align*}
This together with Lemma \ref{lem:8.3} yields
\begin{align*}
    \prod_{h =1}^{r}|\mathcal{F}_{h}| < g(k_{z},t+1)\prod_{h\neq z}{n- t-1 \brack k_{h}-t-1}
    \leq g(k_{r},t+1)\prod_{h=1}^{r-1}{n- t-1 \brack k_{h}-t-1},
\end{align*}
where the second inequality follows from Lemma \ref{lem:8.10}.

\noindent\textbf{Case 2.} $T_{z} \subsetneqq F_{h} \cap T_{\ell} $ for any two distinct $h, \ell \in [r] \setminus \{ z\} $ and $F_{h} \in \mathcal{F}_{h}$. 

Let $h, \ell \in [r] \setminus \{ z\}$ be distinct. By Lemma \ref{lem:3-A} $\rm (iii)$, we have $\dim(F_{z} \cap T_{h}) \geq t$ for each $F_{z} \in \mathcal{F}_{z}$, which implies that
\begin{align*}
    \mathcal{F}_{z} \subseteq \left \{ F \in { V \brack k_{z}} : \dim(F \cap T_{h}) \geq t \right \}.
\end{align*}
Then by (\ref{eq:4.2.*}), we get
\begin{align}\label{eq:5.9.3}
    |\mathcal{F}_{z}|^{r-1} \leq \prod_{h \neq z}g(k_{z},\tau_{t,h}^{\nu^\prime}).
\end{align}
For each $F \in \mathcal{F}_{h}$, observe that $ T_{\ell} \cap T_{z} = T_{z} \subsetneqq F\cap T_{\ell}$, which implies that $\mathcal{F}_{h} = \mathcal{F}_{h}^{>}(T_{\ell},T_{z})$. Hence, by Lemma \ref{lem:10.2} $\rm (i)$ and (\ref{eq:b-f}), we have
\begin{align} \label{eq:5.9.4}
\prod_{h \neq z}|\mathcal{F}_{h}|\leq \left(\prod_{\ell \neq z}\prod_{h \neq \ell,z}\bar{f}(\tau_{t,\ell}^{\nu^\prime},\tau_{t,h}^{\nu^\prime},k_{h},\nu^\prime)\right)^{\frac{1}{r-2}}
    =\prod_{h \neq z}\bar{f}(\tau_{t,h}^{\nu^\prime},\tau_{t,h}^{\nu^\prime},k_{h},\nu^\prime).
\end{align}
By combining this with (\ref{eq:5.9.3}), we obtain that
\begin{align*}
\prod_{h=1}^{r}|\mathcal{F}_{h}|&\leq\prod_{h \neq z}\left(g(k_{z},\tau_{t,h}^{\nu^\prime})\bar{f}(\tau_{t,h}^{\nu^\prime},\tau_{t,h}^{\nu^\prime},k_{h},\nu^\prime)^{r-1}\right)^{\frac{1}{r-1}}.
\end{align*}
This together with Lemma \ref{lem:8.9} yields
\begin{align*}
\prod_{h=1}^{r}|\mathcal{F}_{h}| \leq \prod_{h \neq z}\left(g(k_{z},t+1){n- t-1 \brack k_{h}-t-1}^{r-1}\right)^{\frac{1}{r-1}}
    = g(k_{z},t+1)\prod_{h\neq z}{n- t-1 \brack k_{h}-t-1},
\end{align*}
and then we derive from Lemma \ref{lem:8.10} that 
\begin{align*}
\prod_{h=1}^{r}|\mathcal{F}_{h}| \leq g(k_{r},t+1)\prod_{h=1}^{r-1}{n- t-1 \brack k_{h}-t-1},
\end{align*}
and equality can hold only if $\tau_{t,h}^{\nu^\prime} = t+1\;(h \in [r] \setminus \{ z\})$ and $k_{z}=k_{r}$.

Suppose that equality in \eqref{eq:10.6.A} holds. As Case~1 yields a 
strict inequality, it can only occur in Case~2. Hence all 
inequalities in the preceding chain of Case~2 must be equalities. Applying \eqref{eq:4.2.*}, Lemmas~\ref{lem:8.9} and~\ref{lem:8.10} again yields $\mathcal{F}_z = \left\{ F \in {V \brack k_z} : \dim(F \cap T_h) \ge t \right\}$, $\tau_{t,h}^{\nu'} = t+1$ and $k_z = k_r$ respectively. Since the preceding argument applies to each \(h \in [r] \setminus \{z\}\),  we have $T_{h} = T_{\ell} $ for any two distinct $h, \ell \in [r] \setminus \{ z\}$. Therefore, there exist $T \in {V \brack t+1}$ and $a \in[r]$ with $k_{a}=k_{r}$ such that $\mathcal{F}_{a}=\left\{F \in{V \brack k_{a}}: \dim(T \cap F) \geq t\right\}$ and $\mathcal{F}_{h}=\left\{F \in {V \brack k_{h}}: T \subseteq F\right\}\;(h \in[r] \setminus\{a\})$. 
\end{proof} 

\begin{Lemma} \label{lem:10.3}
 Let $n \geq k_{1}+\nu^\prime+1$. Assume that $\mathcal{F}_{1} \subseteq { V \brack k_{1}}, \mathcal{F}_{2} \subseteq { V \brack k_{2}}, \dots,  \mathcal{F}_{r} \subseteq { V \brack k_{r}}$ are maximal non-trivial $r$-cross $t$-intersecting families satisfying property $\mathcal{P}$. Let $z \in [r]$ with $\tau_{t,z}^{\nu^\prime} \geq t+1$. 
\begin{itemize}
    \item [$\rm (i)$] If $\dim(T \cap F) \geq t+1$ for every $T \in \mathcal{T}_{t}(\mathcal{G}_{z}^{\nu^\prime})$ and $F \in  \cup_{i\in [r] \setminus \{z\}}\mathcal{F}_{i}$, then for each $h\in [r] \setminus\{ z\}$, it holds that $|\mathcal{F}_{h}| \leq  f(\tau_{t,z}^{\nu^\prime},\tau_{t,h}^{\nu^\prime},k_{h},\nu^\prime+1,t+1)$.
    \item [$\rm (ii)$] If there exist $T \in\mathcal{T}_{t}(\mathcal{G}^{\nu^\prime}_{z}), i \in [r] \setminus \{ z\}$ and $F_{i}, F_{i} ^\prime \in \mathcal{F}_{i}$ such that $\dim(T \cap F_{i})=t$ and $\dim(T \cap F_{i} \cap F_{i} ^\prime)<t$, then $|\mathcal{F}_{i}| \leq f(\tau_{t,z}^{\nu^\prime},\tau_{t,i}^{\nu^\prime},k_{i},\nu^\prime,t)$, and $|\mathcal{F}_{h}| < q^{\frac{1}{r}}\bar{f}(\tau_{t,z}^{\nu^\prime},\tau_{t,h}^{\nu^\prime},k_{h},\nu^\prime)$ for each $h\in [r] \setminus\{ z,i\}$.
    \item [$\rm (iii)$] In the remaining case, for each $h\in [r] \setminus\{ z\}$, we have $|\mathcal{F}_{h}| < q^{\frac{1}{r}}\bar{f}(\tau_{t,z}^{\nu^\prime},\tau_{t,h}^{\nu^\prime},k_{h},\nu^\prime)$.
\end{itemize}
\end{Lemma}
\begin{proof}
We note that the existence of $z$ is guaranteed by Lemma \ref{lem:10.5}.

    $\rm (i)$ Let $h \in [r] \setminus \{ z\}$, and fix $T \in \mathcal{T}_{t}(\mathcal{G}_{z}^{\nu^\prime})$. 
    Observe that $\mathcal{G}_{h}^{\nu^\prime}$ is a collection of $t$-covers of $\mathcal{F}_{h}$, and $T$ is a $(t+1)$-cover of $\mathcal{F}_{h}$. It follows from Proposition \ref{proposition:2.4} $\rm (ii)$ that $(\rm i)$ holds.

    $\rm (ii)$ Observe that $\mathcal{G}_{i}^{\nu^\prime}$ is a collection of $t$-covers of $\mathcal{F}_{i}$. By Lemma \ref{lem:3-A} $\rm (iv)$ and Proposition \ref{proposition:2.4} $\rm (i)$, we have 
    \begin{align*}
     |\mathcal{F}_{i}| \leq f(\tau_{t,z}^{\nu^\prime},\tau_{t,i}^{\nu^\prime},k_{i},\nu^\prime,t).   
    \end{align*}
    
    Let $h \in [r] \setminus \{z,i \}$. Observe that $\mathcal{F}_{h} = \mathcal{F}_{h}^{>}(T,F_{i}) \cup \mathcal{F}_{h}^{=}(T,F_{i}) \cup \mathcal{F}_{h}^{<}(T,F_{i})$. Since $\dim(T \cap F_{i})=t$, by Lemma \ref{lem:10.2} $\rm (ii)$, we have $|\mathcal{F}_{h}^{<}(T,F_{i})|<\frac{1}{q^{r+1}}|\mathcal{F}_{h}|$. Also, since $\dim(T \cap F_{i} \cap F_{i} ^\prime)<t$, by Lemma \ref{lem:10.2} $\rm (iii)$, we have $|\mathcal{F}_{h}^{=}(T,F_{i})|<\frac{1}{q^{r+1}}|\mathcal{F}_{h}|$. Then
    \begin{align*}
        |\mathcal{F}_{h}|<\frac{q^{r+1}}{q^{r+1}-2}|\mathcal{F}_{h}^{>}(T,F_{i})|\leq q^{\frac{1}{r}}\bar{f}(\tau_{t,z}^{\nu^\prime},\tau_{t,h}^{\nu^\prime},k_{h},\nu^\prime),
    \end{align*}
    where the second inequality follows from Lemmas \ref{lem:8.2} and \ref{lem:10.2} $\rm (i)$.

     (iii) Since the assumption in (i) does not hold, there exist $T \in\mathcal{T}_{t}(\mathcal{G}^{\nu^\prime}_{z})$, $i \in [r] \setminus\{ z\}$ and $F_{i} \in \mathcal{F}_{i}$ such that $\dim(T \cap F_{i})\leq t$. We observe from Lemma \ref{lem:3-A} $\rm (iii)$ that $\dim(T \cap F_{i})=t$. Next, since the assumption in (ii) does not hold, we have $\dim(T \cap F_{i} \cap F_{i}^\prime) \geq t$ for each $F_{i}^\prime \in \mathcal{F}_{i}$. Thus $\mathcal{F}_{i} = \mathcal{F}_{i}^{>}(T,F_{i}) \cup \mathcal{F}_{i}^{=}(T,F_{i})$. Since $\tau_{t,z}^{\nu^\prime} \geq t+1$, there exist $j \in [r] \setminus\{ z,i\} $ and $F_{j}^\prime \in \mathcal{F}_{j}$ such that $\dim(T \cap F_{i} \cap F_{j}^\prime)<t$. By Lemma \ref{lem:10.2} $\rm (iii)$, we have $|\mathcal{F}_{i}^{=}(T,F_{i})|<\frac{1}{q^{r+1}}|\mathcal{F}_{i}|$. Then
    \begin{align*}
        |\mathcal{F}_{i}|<\frac{q^{r+1}}{q^{r+1}-1}|\mathcal{F}_{i}^{>}(T,F_{i})|\leq q^{\frac{1}{r}}\bar{f}(\tau_{t,z}^{\nu^\prime},\tau_{t,i}^{\nu^\prime},k_{i},\nu^\prime).
    \end{align*}

    We claim that $T \cap F \neq T \cap F_{i}$ for each $F \in \mathcal{F}_{j}$. Indeed, if there exists $F_{j} \in \mathcal{F}_{j}$ such that $T \cap F_{j} = T \cap F_{i}$, then $\dim(T \cap F_{j})=t$ and $\dim(T \cap F_{j}\cap F_{j}^\prime)<t$, implying that the conditions of part $\rm (ii)$ are satisfied, which contradicts the current assumption. Thus $\mathcal{F}_{j} = \mathcal{F}_{j}^{>}(T,F_{i}) \cup \mathcal{F}_{j}^{<}(T,F_{i})$. It follows from $\dim(T \cap F_{i})=t$ and Lemma \ref{lem:10.2} $\rm (ii)$ that $|\mathcal{F}_{j}^{<}(T,F_{i})|<\frac{1}{q^{r+1}}|\mathcal{F}_{j}|$.  Then
    \begin{align*}
        |\mathcal{F}_{j}|<\frac{q^{r+1}}{q^{r+1}-1}|\mathcal{F}_{j}^{>}(T,F_{i})|\leq q^{\frac{1}{r}}\bar{f}(\tau_{t,z}^{\nu^\prime},\tau_{t,j}^{\nu^\prime},k_{j},\nu^\prime).
    \end{align*}

    As the families involving \(\mathcal{F}_{i}\) and \(\mathcal{F}_{j}\) have already been handled, the proof is complete when \(r=3\). Now we assume that \(r\ge 4\). Let $h \in [r] \setminus \{ z,i,j\}$. Observe that $\mathcal{F}_{h} = \mathcal{F}_{h}^{>}(T,F_{i}) \cup \mathcal{F}_{h}^{=}(T,F_{i}) \cup \mathcal{F}_{h}^{<}(T,F_{i})$. Since $\dim(T \cap F_{i})=t$, by Lemma \ref{lem:10.2} $\rm (ii)$, we have $|\mathcal{F}_{h}^{<}(T,F_{i})|<\frac{1}{q^{r+1}}|\mathcal{F}_{h}|$. Also, from $\dim(T \cap F_{i} \cap F_{j}^\prime)<t$ and Lemma \ref{lem:10.2} $\rm (iii)$, we obtain $|\mathcal{F}_{h}^{=}(T,F_{i})|<\frac{1}{q^{r+1}}|\mathcal{F}_{h}|$. Then
    \begin{align*}
        |\mathcal{F}_{h}|<\frac{q^{r+1}}{q^{r+1}-2}|\mathcal{F}_{h}^{>}(T,F_{i})|\leq q^{\frac{1}{r}}\bar{f}(\tau_{t,z}^{\nu^\prime},\tau_{t,h}^{\nu^\prime},k_{h},\nu^\prime).
    \end{align*}
    This completes the proof.
\end{proof}

\begin{Lemma} \label{lem:10.4}
 Let $n \geq k_{1}+\nu^\prime+3$. Let $\mathcal{F}_{1} \subseteq { V \brack k_{1}}, \mathcal{F}_{2} \subseteq { V \brack k_{2}}, \dots, \mathcal{F}_{r} \subseteq { V \brack k_{r}}$ be maximal non-trivial $r$-cross $t$-intersecting families satisfying property $\mathcal{P}$. If $\tau_{t,z}^{\nu^\prime}\geq t+1$ for all $z \in [r]$, then 
\begin{align} \label{eq:S-A}
    \prod_{h=1}^{r}|\mathcal{F}_{h}|
    < g(k_{r},t+1)\prod_{h=1}^{r-1}{n- t-1 \brack k_{h}-t-1}.
\end{align}
\end{Lemma}
\begin{proof}
 For each $z \in [r]$, let $V_{z}$ be the set of  $i \in [r] \setminus \{ z\}$ with the property that there exist $T \in\mathcal{T}_{t}(\mathcal{G}^{\nu^\prime}_{z})$ and $F_{i}, F_{i} ^\prime \in \mathcal{F}_{i}$ satisfying $\dim(T \cap F_{i})=t$ and $\dim(T \cap F_{i} \cap F_{i} ^\prime)<t$. Now we divide the proof into the following two cases.

\noindent\textbf{Case 1.} For each $i \in [r]$, there exists $i^\prime \in [r]$ such that $V_{i^\prime} = \{i\}$ .

In this case, observe that $i^\prime$ is unique for each $i \in [r]$, and $i_{1}^\prime \neq i_{2}^\prime$ for any two distinct $i_{1},i_{2} \in [r]$. By Lemma \ref{lem:10.3} $\rm (ii)$ and (\ref{eq:b-f}), we have
\begin{align*}
    \prod_{i=1}^{r}\prod_{h \neq i,i^\prime}|\mathcal{F}_{h}| \leq \prod_{i=1}^{r}q^{\frac{r-2}{r}}\prod_{h \neq i,i^\prime}\bar{f}(\tau_{t,i^\prime}^{\nu^\prime},\tau_{t,h}^{\nu^\prime},k_{h},\nu^\prime)
    = q^{r-2}\left( \prod_{ h =1}^{r}\bar{f}(\tau_{t,h}^{\nu^\prime},\tau_{t,h}^{\nu^\prime},k_{h},\nu^\prime)\right)^{r-2}.
\end{align*}
It follows from Lemma \ref{lem:8.6} that
\begin{align*}
\left(\prod_{h=1}^{r}|\mathcal{F}_{h}|\right)^{r-2} \leq q^{r-2}\left( \prod_{ h =1}^{r}{n- t-1 \brack k_{h}-t-1}\right)^{r-2} < \left( g(k_{r},t+1)\prod_{h=1}^{r-1}{n- t-1 \brack k_{h}-t-1}\right)^{r-2},
\end{align*}
where the second inequality follows from Lemma \ref{lem:8.3}. Consequently, (\ref{eq:S-A}) holds.

\noindent\textbf{Case 2.} There exists $i \in [r]$ such that $V_{z} \neq\{i\}$ for each $z \in [r]$.

We construct inductively a permutation of $[r]$. Set $a_1 = i$. For $k = 1, 2, \dots, r-2$, having already chosen pairwise distinct elements $a_1, a_2, \dots, a_k$, we define $a_{k+1}$ as follows:
\begin{itemize}
    \item [$\rm (i)$] if $V_{a_k}$ is not a singleton, pick 
          $a_{k+1}$ arbitrarily from 
          $[r] \setminus \{a_1, a_2, \dots, a_k\}$;
    \item [$\rm (ii)$] if $V_{a_k}$ is a singleton, pick $a_{k+1}$ arbitrarily from $[r] \setminus \bigl(\{a_1, a_2, \dots, a_k\} \cup V_{a_k}\bigr)$.
\end{itemize}
After completing the first $r-2$ steps, the set 
$[r] \setminus \{a_1, a_2, \dots, a_{r-1}\}$ contains exactly one 
element, and we call it $a_{r}$. The resulting $r$-cycle $(a_1 \; a_2 \; \dots \; a_{r})$ 
is the desired permutation. Set $a_{r+1}=a_{1}$. Observe that $V_{a_{h}} \neq\{a_{h+1}\}$ for each $h \in [r] \setminus \{ r-1\}$, whereas we may have \(V_{a_{r-1}} = \{a_{r}\}\).

We proceed by proving a claim.

\noindent{\bf Claim.}\;We have 
\begin{equation*}
     |\mathcal{F}_{a_{h+1}}|< q^{\frac{1}{r}}f(\tau_{t,a_{h}}^{\nu^\prime},\tau^{\nu^\prime}_{t,a_{h+1}},k_{a_{h+1}},\nu^\prime+1,t+1)
 \end{equation*}
 for each $h \in [r] $ with $V_{a_{h}} \neq\{a_{h+1}\}$. 
 
Since $\tau_{t,a_{h}}^{\nu^\prime} \geq t+1$, we have ${\tau_{t,a_{h}}^{\nu^\prime}-t \brack 1} \leq {\tau_{t,a_{h}}^{\nu^\prime} \brack t+1}$, which implies that 
 \begin{align*} 
     \bar{f}(\tau_{t,a_{h}}^{\nu^\prime},\tau^{\nu^\prime}_{t,a_{h+1}},k_{a_{h+1}},\nu^\prime) \leq f(\tau_{t,a_{h}}^{\nu^\prime},\tau^{\nu^\prime}_{t,a_{h+1}},k_{a_{h+1}},\nu^\prime+1,t+1).
 \end{align*}
If $V_{a_{h}} =\varnothing$, then Lemma \ref{lem:10.3} $\rm (i)$ and $\rm (iii)$ yield
\begin{align*}
    |\mathcal{F}_{a_{h+1}}| \leq \max\left\{f(\tau_{t,a_{h}}^{\nu^\prime},\;\tau^{\nu^\prime}_{t,a_{h+1}},k_{a_{h+1}},\nu^\prime+1,t+1),q^{\frac{1}{r}}\bar{f}(\tau_{t,a_{h}}^{\nu^\prime},\tau^{\nu^\prime}_{t,a_{h+1}},k_{a_{h+1}},\nu^\prime) \right\},
\end{align*}
and the desired inequality is true. Assume that $V_{a_{h}} \neq\varnothing$. Then there exists $j \in [r] \setminus \{ a_{h},a_{h+1}\}$ such that $j \in V_{a_{h}}$. By Lemma \ref{lem:10.3} $\rm (ii)$, we have
\begin{align*}
    |\mathcal{F}_{a_{h+1}}| <q^{\frac{1}{r}}\bar{f}(\tau_{t,a_{h}}^{\nu^\prime},\tau^{\nu^\prime}_{t,a_{h+1}},k_{a_{h+1}},\nu^\prime)\leq q^{\frac{1}{r}}f(\tau_{t,a_{h}}^{\nu^\prime},\tau^{\nu^\prime}_{t,a_{h+1}},k_{a_{h+1}},\nu^\prime+1,t+1),
\end{align*}
as claimed.
 
Suppose $V_{a_{r-1}}\neq \{ a_{r} \}$. By the claim and (\ref{eq:f}), we have
\begin{align*}
    \prod_{h=1}^{r}|\mathcal{F}_{h}| &< \prod_{h=1}^{r}\left(q^{\frac{1}{r}}f(\tau_{t,a_{h}}^{\nu^\prime},\tau^{\nu^\prime}_{t,a_{h+1}},k_{a_{h+1}},\nu^\prime+1,t+1)\right)=q\left(\prod_{h=1}^{r}f(\tau^{\nu^\prime}_{t,h},\tau_{t,h}^{\nu^\prime},k_{h},\nu^\prime+1,t+1)\right).
\end{align*}
This together with Lemma \ref{lem:8.4} yields 
\begin{align*}
    \prod_{h=1}^{r}|\mathcal{F}_{h}| &< q \prod_{ h =1}^{r}{n- t-1 \brack k_{h}-t-1}<g(k_{r},t+1)\prod_{h=1}^{r-1}{n- t-1 \brack k_{h}-t-1},
\end{align*}
where the second inequality follows from Lemma \ref{lem:8.3}.

Now suppose $V_{a_{r-1}}= \{ a_{r} \}$. By Lemma \ref{lem:10.3}, our claim  and (\ref{eq:f}), we have
\begin{align*}
    \prod_{h=1}^{r}|\mathcal{F}_{h}| <&f(\tau_{t,a_{r-1}}^{\nu^\prime},\tau_{t,a_{r}}^{\nu^\prime},k_{a_{r}},\nu^\prime,t)\prod_{h \neq r-1}\left(q^{\frac{1}{r}}f(\tau_{t,a_{h}}^{\nu^\prime},\tau^{\nu^\prime}_{t,a_{h+1}},k_{a_{h+1}},\nu^\prime+1,t+1)\right)\\
    =&q^{1-\frac{1}{r}}\cdot f(\tau^{\nu^\prime}_{t,a_{r-1}},\tau^{\nu^\prime}_{t,a_{r-1}},k_{a_{r-1}},\nu^\prime,t)\prod_{h \neq a_{r-1}}f(\tau^{\nu^\prime}_{t,h},\tau^{\nu^\prime}_{t,h},k_{h},\nu^\prime+1,t+1).
\end{align*}
This together with Lemmas \ref{lem:8.4} and \ref{lem:8.3} yields
\begin{align*}
    \prod_{h=1}^{r}|\mathcal{F}_{h}| &< g\left(k_{a_{r-1}},t+1\right) \prod_{ h \neq a_{r-1}}{n- t-1 \brack k_{h}-t-1}\leq g(k_{r},t+1)\prod_{h=1}^{r-1}{n- t-1 \brack k_{h}-t-1},
\end{align*}
where the second inequality follows from Lemma \ref{lem:8.10}.
\end{proof}

\noindent {\bf Proof of Theorem \ref{theorem:4}.}\;We may assume that $\mathcal{F}_1,\mathcal{F}_{2},\ldots,\mathcal{F}_r$ are maximal. Recall that the parameter $\nu'$ is defined in (\ref{equdefnu'}). Also recall that for the families $\mathcal{F}_1,\mathcal{F}_{2},\ldots,\mathcal{F}_r$, the parameter $m$ is defined in (\ref{equdefm}), and we say that they satisfy property $\mathcal{P}$ if $m \leq \nu^\prime$ and (\ref{equpropertyp}) holds for each pair of distinct $i,j \in [r]$. By the definition of $\nu'$ and our assumption on $n$, we have $n\geq k_1+\nu'+4.$

If $\mathcal{F}_{1},\mathcal{F}_{2}, \dots,\mathcal{F}_{r}$ do not satisfy property $\mathcal{P}$, then Lemmas \ref{lem:10.10} and \ref{lem:10.17} give the strict inequality
\begin{align*}
    \prod_{ h =1}^{r}|\mathcal{F}_{h}| 
    < g(k_{r},t+1)\prod_{h =1}^{r-1}{n-t-1 \brack k_{h}-t-1}.
\end{align*}
Next, suppose that $\mathcal{F}_{1},\mathcal{F}_{2}, \dots,\mathcal{F}_{r}$ satisfy property $\mathcal{P}$. Note that $n\geq k_1+\nu'+4$ implies $n \geq\max\{k_{1}+t+1,k_{1}+\nu^\prime-t+3+\frac{t+2}{r-1}\}$, and so the assumptions of Lemmas \ref{lem:10.6} and \ref{lem:10.4} are satisfied. So by these two lemmas, we obtain
\begin{align*}
    \prod_{ h =1}^{r}|\mathcal{F}_{h}| 
    \leq g(k_{r},t+1)\prod_{h =1}^{r-1}{n-t-1 \brack k_{h}-t-1},
\end{align*}
and equality holds if and only if there exist $T \in {V \brack t+1}$ and $a \in[r]$ with $k_{a}=k_{r}$ such that $\mathcal{F}_{a}=\left\{F \in{V \brack k_{a}}: \dim(T \cap F) \geq t\right\}$ and $\mathcal{F}_{h}=\left\{F \in {V \brack k_{h}}: T \subseteq F\right\}\;(h \in[r] \setminus\{a\})$.\qed

\section*{Acknowledgments}
B. Lv is supported by the National Natural Science Foundation of China (12571347 \& 12131011), and the Beijing Natural Science Foundation (1252010).
\
\addcontentsline{toc}{chapter}{Bibliography}

{
	}
\end{document}